\documentclass[11pt]{amsart}
\usepackage{amsmath, amsfonts, amsbsy, amssymb, upref, chatterjee1}

\begin{document}
%\begin{titlepage}
\title[Disorder chaos and multiple valleys]{Disorder chaos and multiple valleys in spin glasses}
\author{Sourav Chatterjee}
\date{September 5, 2009}
\thanks{The author's research was partially supported by NSF grant DMS-0707054 and a Sloan Research Fellowship}
\address{\newline367 Evans Hall \#3860\newline
Department of Statistics\newline
University of California at Berkeley\newline
Berkeley, CA 94720-3860\newline
{\it E-mail: \tt sourav@stat.berkeley.edu}\newline 
{\it URL: \tt http://www.stat.berkeley.edu/$\sim$sourav}}

\begin{abstract}
We prove that the Sherrington-Kirkpatrick model of spin glasses is chaotic under small perturbations of the couplings at any temperature in the absence of an external field. The result is proved for two kinds of perturbations: (a) distorting the couplings via Ornstein-Uhlenbeck flows,  and (b) replacing a small fraction of the couplings by independent copies. We further prove that the S-K model exhibits multiple valleys in its energy landscape, in the weak sense that there are many states with near-minimal energy that are mutually nearly orthogonal. We show that the variance of the free energy of the S-K model is unusually small at any temperature. (By `unusually small' we mean that it is much smaller than the number of sites; in other words, it beats the classical  Gaussian concentration inequality, a phenomenon that we call `superconcentration'.) We prove that the bond overlap in the Edwards-Anderson model of spin glasses is {\it not} chaotic under perturbations of the couplings, even large perturbations. Lastly, we obtain sharp lower bounds on the variance of the free energy in the E-A model on any bounded degree graph, generalizing a result of Wehr and Aizenman and establishing the absence of superconcentration in this class of models. Our techniques apply for the $p$-spin models and the Random Field Ising Model as well, although we do not work out the details in these cases. 
\end{abstract}
\maketitle
%\end{titlepage}

\tableofcontents

\section{Introduction}\label{intro}
Spin glasses are magnetic materials with strange properties that distinguish them from ordinary ferromagnets. In statistical physics, the study of spin glasses originated with the works of Edwards and Anderson~\cite{edwardsanderson75} and Sherrington and Kirkpatrick \cite{sk75} in 1975. In the following decade, the theoretical study of spin glasses led to the invention of deep and powerful new  methods in physics, most notably  Parisi's broken replica method. We refer to  \cite{mezardetal87} for a survey of the physics literature.

However, these physical breakthroughs were far beyond the reach of rigorous proof at the time, and much of it remains so till date. The rigorous analysis of the Sherrington-Kirkpatrick model began with the works of Aizenman, Lebowitz and Ruelle \cite{alr87} and Fr\"ohlich and Zegarli\'nski \cite{frohlichzegarlinski87} in the late eighties; the field remained stagnant for a while, interspersed with a few nice papers occasionally (e.g.~\cite{cometsneveu95}, \cite{shcherbina97}). The deepest mysteries of the broken replica analysis of the S-K model remained mathematically intractable for many more years until the path-breaking contributions of Guerra, Toninelli, Talagrand, Panchenko and others in the last ten years (see e.g.\  \cite{arguinaizenman09}, \cite{guerra02}, \cite{ guerra03}, \cite{ panchenkotalagrand07}, \cite{ghirlandaguerra98}, \cite{talagrand03}, \cite{talagrand06}). Arguably the most notable achievement in this period was  Talagrand's proof of the Parisi formula \cite{talagrand06}.

However, in spite of all this remarkable progress, our understanding of these complicated mathematical objects is still shrouded in mystery, and many conjectures remain unresolved. In this article we attempt to give a mathematical foundation to some aspects of spin glasses  that have been well-known in the physics community for a long time but never before penetrated by rigorous mathematics. Let us now embark on a description of our main results. Further references and connections with the literature will be given at the appropriate places along the way.

\subsection{Weak multiple valleys in the S-K model}
Consider the following simple-looking probabilistic question: Suppose $(g_{ij})_{1\le i,j\le N}$ are i.i.d.\ standard Gaussian random variables, and we define, for each $\bos \in \{-1,1\}^N$, the quantity 
\begin{equation}\label{xn}
X_N(\bos) := \sum_{1\le i,j\le N} g_{ij}\sigma_i \sigma_j. 
\end{equation}
Then is it true that with high probability, there is a large subset $A$ of $\{-1,1\}^N$ such that 
\begin{equation}\label{xmax}
X_N(\bos) \simeq \max_{\bos'\in\{-1,1\}^N} X_N(\bos') \ \ \text{for each $\bos \in A$,}
\end{equation}
and any two distinct elements $\bos^1, \bos^2$ of $A$ are nearly orthogonal, in the sense that
\begin{equation}\label{r12}
R_{\bos^1,\bos^2} = R_{1,2} := \frac{\sum_{i=1}^N \sigma_i^1 \sigma_i^2}{N} \simeq 0? 
\end{equation}
(In the spin glass literature, the quantity $R_{1,2}$ is called the `overlap' between the `configurations' $\bos^1$ and $\bos^2$.) To realize the non-triviality of the question, consider a slightly different Gaussian field $Y_N$ on $\{-1,1\}^N$,  defined as
\[
Y_N(\bos) := \sum_{i=1}^N g_i \sigma_i,
\]
where $g_1,\ldots, g_N$ are i.i.d.\ standard Gaussian random variables. Then clearly, $Y_N$ is maximized at $\hat{\bos}$, where $\hat{\sigma}_i = \mathrm{sign}(Y_i)$. Note that for any $\bos$,
\[
Y_N(\bos) = \sum_{i: \ \sigma_i = \hat{\sigma}_i } |Y_i| - \sum_{i: \ \sigma_i = - \hat{\sigma}_i}|Y_i|. 
\]
It is not difficult to argue from here that if $\bos$ is another configuration that is near-maximal for $Y_N$, then $\bos$ must agree with $\hat{\bos}$ at nearly all coordinates. Thus, the field $Y_N$ does not satisfy the `multiple peaks picture' that we are investigating about $X_N$. This is true in spite of the fact that $Y_N(\bos)$ and $Y_N(\bos')$ are approximately independent for almost all pairs $(\bos, \bos')$. %The reason for the lack of multiple peaks is that the independence is not weak enough to override the large size of~$\{-1,1\}^N$. 

We have the following result about the existence of multiple peaks in the field $X_N$. It says that with high probability, there is a large collection $A$ of configurations satisfying \eqref{xmax} and \eqref{r12}, that is, $R_{\bos^1,\bos^2}\simeq 0$ for any two distinct $\bos^1,\bos^2\in A$, and $X_N(\bos)\simeq \max_{\bos'}X_N(\bos')$ for each $\bos\in A$. 
\begin{thm}\label{multisk}
Let $X_N$ be the field defined in \eqref{xn}, and define the overlap $R_{\bos^1,\bos^2}$ between configurations $\bos^1, \bos^2$ by the formula \eqref{r12}. Let 
\[
M_N := \max_{\bos} X_N(\bos).
\]
Then there are constants $r_N \ra \infty$, $\gamma_N \ra 0$, $\epsilon_N \ra 0$, and $\delta_N\ra 0$ such that with probability at least $1-\gamma_N$, there is a set $A\subseteq\{-1,1\}^N$ satisfying
\begin{enumerate}
\item[(a)] $|A| \ge r_N$,
\item[(b)] $R_{\bos^1,\bos^2}^2\le \epsilon_N$ for all $\bos^1, \bos^2\in A$, $\bos^1\ne \bos^2$, and 
\item[(c)] $X_N(\bos) \ge (1-\delta_N)M_N$ for all $\bos\in A$. 
\end{enumerate}
Quantitatively, we can take $r_N = (\log N)^{1/8}$, $\delta_N = (\log N)^{-1/8}$, $\epsilon_N = e^{-(\log N)^{1/8}}$ and $\gamma_N = C(\log N)^{-1/12}$, where $C$ is an absolute constant. However these are not necessarily the best choices. 
\end{thm}
Let us now discuss the implication of this result in spin glass theory. The Sherrington-Kirkpatrick model of spin glasses, introduced in \cite{sk75}, is defined through the Hamiltonian (i.e.\ energy function)
\begin{equation}\label{skhamil}
H_N(\bos) := -\frac{1}{\sqrt{2N}} X_N(\bos) = -\frac{1}{\sqrt{2N}} \sum_{1\le i,j\le N} g_{ij} \sigma_i \sigma_j.
\end{equation}
The S-K model at inverse temperature $\beta\ge 0$ defines a probability measure  $G_N$ on $\{-1,1\}^N$ through the formula
\begin{equation}\label{gibbs}
G_N(\{\bos\}) := Z(\beta)^{-1}e^{-\beta H_N(\bos)},
\end{equation}
where $Z(\beta)$ is the normalizing constant. The measure $G_N$ is called the Gibbs measure. 

According to the folklore in the statistical physics community, the energy landscape of the S-K model has `multiple valleys'. Although no precise formulation is available, one way to view this is that there are many nearly orthogonal states with nearly minimal energy. For a physical discussion of the `many states' aspect of the S-K model, we refer to \cite{mezardetal87}, Chapter III. A~very interesting rigorous formulation was attempted by Talagrand~(see \cite{talagrand03}, Conjecture 2.2.23), but no theorems were proved.  Although our achievement is quite modest, and may not be satisfactory to the physicists because we do not prove that the approximate minimum energy states correspond to significantly large regions of the state space --- in fact, one may say that it is not what is meant by the physical term `multiple valleys' at all because an isolated low energy state does not necessarily  represent a valley --- it does seem that Theorem~\ref{multisk} is the first rigorous result about the multimodal geometry of the Sherrington-Kirkpatrick energy landscape. We may call it `multiple valleys in a weak sense'. 

Theorem \ref{multisk} can be generalized to the following Corollary, which shows that weak multiple valleys exist at `every energy level' and not only for the lowest energy.
\begin{cor}\label{multisk2}
Let all notation be the same as in Theorem \ref{multisk}. Fix a number $\alpha \in (0,1]$. Then for all sufficiently large $N$, with probability at least $1-2\gamma_N$ there exists a set $A\subseteq \{-1,1\}^N$ satisfying conditions \textup{(a)} and \textup{(b)} of Theorem~\ref{multisk}, such that $|X_N(\bos) - \alpha M_N|\le \delta_N|M_N|$ for all $\bos \in A$.
\end{cor} 
%When a configuration $\bos$ is chosen from the Gibbs measure $G_N$ at inverse temperature $\beta$, it is presumably the case (but not yet proven), by the general principle of equivalence of ensembles, that $X_N(\bos) \approx \alpha M_N$ for some $\alpha\in [0,1]$ depending only on $\beta$. 
%[An aside for experts on spin glasses: It can be shown rigorously (e.g.\ using Theorem~2.4.17 of~\cite{talagrand03}) that if $\bos^1$ and $\bos^2$ are independent picks from the Gibbs measure at inverse temperature $\beta$, then $\ee(R_{1,2}^2) \ge 1 - C/\beta$ where $C$ is some positive absolute  constant. In some sense, this fact is in striking contrast with Corollary \ref{multisk2} when $\beta$ is large. The only resolution of this paradox must be that the Gibbs measure `sees' only a very small region of the state space. This insightful observation was communicated to the author by Michel Talagrand.] %Note that this is in striking contrast to Corollary \ref{multisk2}, in some sense --- an observation communicated to the author by Michel Talagrand. %
%, that Theorem \ref{multisk} stands in contrast to the fact that for $\beta$ sufficiently large, $\ee(R_{1,2}^2) \ge K/\beta$ for some absolute constant~$K$. 

The variables $(g_{ij})_{1\le i,j\le N}$ in the Hamiltonian $H_N$ are collectively called the `couplings' or the `disorder'. Our proof of Theorem \ref{multisk} is based on the chaotic nature of the S-K model under small perturbations of the couplings; this is discussed in the next subsection. The relation between chaos and multiple valleys follows from a general principle outlined in~\cite{chatterjee08}, although the proof in the present  paper is self-contained. 

\subsection{Disorder chaos in the S-K model}\label{chaossec}
Recall the Gibbs measure $G_N$ of the S-K model, defined in \eqref{gibbs}. Suppose $\bos^1$ and $\bos^2$ are two configurations drawn independently according to the measure $G_N$, and the overlap $R_{1,2}$ is defined as in \eqref{r12}. It is known that when $\beta < 1$, $R_{1,2}\simeq 0$ with high probability \cite{frohlichzegarlinski87, cometsneveu95, talagrand03}. However, it is also known that $R_{1,2}$ cannot be concentrated near zero for all $\beta$, because that  would give a contradiction to the existence of a phase transition as established in \cite{alr87}. In fact, it is believed that the limiting distribution of $R_{1,2}$ in the low temperature phase is given by the so-called `Parisi measure', a notion first made rigorous by Talagrand \cite{talagrand06, talagrand06a}.

Now suppose we choose $\bos^2$ not from the Gibbs measure $G_N$, but from a new Gibbs measure $G_N'$, based on a new Hamiltonian $H_N'$ which is obtained by applying a small perturbation to the Hamiltonian $H_N$. (We will make precise the notion of a small perturbation below.) Is it still true that $R_{1,2}$ has a non-degenerate limiting distribution at low temperatures? The conjecture of disorder chaos (i.e.\ chaos with respect to small fluctuations in the disorder $(g_{ij})_{1\le i,j\le N}$) states that indeed that is not the case: $R_{1,2}$ is concentrated near zero if $\bos^1$ is picked from the Gibbs measure and $\bos^2$ is picked from a perturbed Gibbs measure. This is supposed to be true at all temperatures. To the best of our knowledge, disorder chaos for the S-K model was first discussed in the widely cited paper of Bray and Moore \cite{braymoore87}; a related discussion appears in the earlier paper~\cite{mckayetal82}. The phenomenon of chaos itself was first conjectured by Fisher and Huse \cite{fisherhuse86} in the context of the Edwards-Anderson model, although the term was coined in \cite{braymoore87}. Again, to the best of our knowledge, nothing has been proved rigorously yet. For further references in the physics literature, let us refer to the recent paper~\cite{kk07}. 

%For further pointers to the physics literature, we refer to the survey in \cite{chatterjee08}.

Note that this idea of chaos should not be confused with temperature chaos (also discussed in \cite{braymoore87}), which says that spin glasses are chaotic with respect to small changes in the inverse temperature $\beta$. 

We shall consider two kinds of perturbation of the disorder. The first, what we call `discrete perturbation', is executed by replacing a randomly chosen small fraction of the couplings $(g_{ij})$ by independent copies. Here small fraction means a fraction $p$ that goes to zero as $N\ra \infty$. Discrete perturbation is the usual way to proceed in the noise-sensitivity literature (see e.g.\ \cite{bks99, bks03, schrammsteif05, mosseletal05, gps08}). In fact, it seems that the following result is intimately connected with noise-sensitivity, although we do not see any obvious way to use the standard noise-sensitivity techniques to derive it.
\begin{thm}\label{chaosdisc}
Consider the S-K model at inverse temperature $\beta$. Take any $N$ and $p\in [0,1]$. Suppose a randomly chosen fraction $p$ of the couplings $(g_{ij})$ are replaced by independent copies to give a perturbed Gibbs measure. Let $\bos^1$ be chosen from the original Gibbs measure and $\bos^2$ is chosen from the perturbed measure. Let the overlap $R_{1,2}$ be defined as in~\eqref{r12}. Then 
\[
\ee(R_{1,2}^2) \le \frac{C(1+\beta)}{p\log N},
\]
where $C$ is an absolute constant and the expectation is taken over all randomness. 
\end{thm} 
This theorem shows that the system is chaotic if the fraction $p$ goes to zero slower than $1/\log N$. The derivation of this result is based on the `superconcentration' property of the free energy in the S-K model that we present in the next subsection. 

The notion of perturbation in the above theorem, though natural, is not the only available notion. In fact, in the original physics papers (e.g.\ \cite{braymoore87}), a different manner of perturbation is proposed, which we call continuous perturbation. Here we replace $g_{ij}$ by $a g_{ij} + bg_{ij}'$, where $(g_{ij}')$ is another set of indepenent standard Gaussian random variables and  $a^2 + b^2 = 1$ so that the resultant couplings are again standard Gaussian. When $a\simeq 1$, we say that the perturbation is small. A convenient way to parametrize the perturbation is to set $a = e^{-t}$, where $t$ is a parameter that we call `time'. This nomenclature is natural, because perturbing the couplings up to time $t$ corresponds to running an Ornstein-Uhlenbeck flow at each coupling for  time $t$, with initial value $g_{ij}$. The following theorem says that the S-K model is chaotic under small continuous perturbations. 
\begin{thm}\label{chaoscont}
Consider the S-K model at inverse temperature $\beta$. Take any $t \ge 0$.  Suppose we continuously perturb the couplings up to time~$t$, as defined above. Let $\bos^1$ be chosen from the original Gibbs measure and $\bos^2$ be chosen from the perturbed measure. Let the overlap $R_{1,2}$ be defined as in~\eqref{r12}. Then there is an absolute constant $C$ such that for any positive integer $k$, 
\[
\ee(R_{1,2}^{2k}) \le (Ck)^k N^{- k \min\{1,\; t/C\log (1+C\beta)\}}.
\]
The expectation is taken over all randomness. 
\end{thm}
Again, the achievement is very modest, and does not come anywhere close to the claims of the physicists. But once again, this is the first rigorous result about chaos of any kind in the S-K model. To the best of our knowledge, the only other instance of a rigorous proof of chaos in any spin glass model is in the work of Panchenko and Talagrand~\cite{panchenkotalagrand07}, who established chaos with respect to small changes in the external field in the spherical S-K model. Disorder chaos in directed polymers was established by the author in \cite{chatterjee08}. 

A deficiency of both theorems in this subsection is that they do not cover the case of zero temperature, that is, $\beta=\infty$, where  Gibbs measure concentrates all its mass on the ground state. In principle, the same techniques should apply, but there are some crucial hurdles that cannot be cleared with the available ideas. 

\subsection{Superconcentration in the S-K model} The notion of superconcentration was defined in~\cite{chatterjee08}. The definition in \cite{chatterjee08} pertains only to maxima of Gaussian fields, but it can be generalized to roughly mean the following: a Lipschitz function of a collection of independent standard Gaussian random variables is superconcentrated whenever its order of fluctuations is much smaller than its Lipschitz constant. This definition is related to the classical concentration result for the Gaussian measure, which says that the order of fluctuations of a Lipschitz function under the Gaussian measure is bounded by its Lipschitz constant (see e.g.\ Theorem 2.2.4 in \cite{talagrand03}), irrespective of the dimension. 

The free energy of the S-K model is defined as
\begin{equation}\label{free}
F_N(\beta) := \frac{1}{\beta}\log \sum_{\bos\in \{-1,1\}^N} e^{-\beta H_N(\bos)},
\end{equation}
where $H_N$ is the Hamiltonian defined in \eqref{skhamil}. It follows from classical concentration of measure that the variance of $F_N(\beta)$ is bounded by a constant multiple of $N$ (see Corollary 2.2.5 in \cite{talagrand03}). This is the best known bound for $\beta > 1$. When $\beta < 1$, Talagrand (Theorems 2.2.7 and 2.2.13 in \cite{talagrand03}) proved that the variance can actually be bounded by an absolute constant. This is also indicated in the earlier works of Aizenman, Lebowitz and Ruelle \cite{alr87} and Comets and Neveu \cite{cometsneveu95}. Therefore, according to our definition, the free energy is superconcentrated when $\beta < 1$. The following theorem shows that $F_N$ is superconcentrated at any $\beta$. 
\begin{thm}\label{superconc}
Let $F_N(\beta)$ be the free energy of the S-K model defined above in \eqref{free}. For any $\beta$, we have
\[
\var F_N(\beta)\le \frac{C N\log (2+C\beta)}{\log N},
\]
where $C$ is an absolute constant. 
\end{thm}
This result may be reminiscent of the $\log N$ improvement in the variance of first  passage percolation time~\cite{bks03}. However, the proof is quite different in our case since hypercontractivity, the major tool in \cite{bks03}, does not seem to work for spin glasses in any obvious way. In that sense, the two results are quite unrelated. Our proof is based on our  chaos theorem for continuous perturbation (Theorem \ref{chaoscont}) and ideas from \cite{chatterjee08}.  On the other hand, Theorem \ref{superconc} is used to derive the chaos theorem for discrete perturbation, again drawing upon ideas from \cite{chatterjee08}. This equivalence between chaos and superconcentration is one of the main themes of \cite{chatterjee08}, which in a way shows the significance of superconcentration, which may otherwise be viewed as just a curious phenomenon.  

Incidentally, it was shown by Talagrand (\cite{talagrand07}, eq.\ (10.13)) that the lower tail fluctuations of $F_N(\beta)$ are actually as small as order $1$ under an unproven hypothesis about the Parisi measure. 

\subsection{Disorder chaos in the E-A model}\label{ea}
Let $G = (V, E)$ be an undirected graph. The Edwards-Anderson spin glass \cite{edwardsanderson75} on $G$ is defined through the Hamiltonian
\begin{equation}\label{eahamil}
H(\bos) := - \sum_{(i,j)\in E} g_{ij} \sigma_i\sigma_j,\ \ \bos\in \{-1,1\}^V,
\end{equation}
where $(g_{ij})$ is again a collection of i.i.d.\ random variables, often taken to be Gaussian. The S-K model corresponds to the case of the complete graph, up to normalization by $\sqrt{N}$. 

For a survey of the (few) rigorous and non-rigorous results available for the Edwards-Anderson model, we refer to Newman and Stein \cite{newmanstein07}.

Unlike the S-K model, there are two kinds of overlap in the E-A model. The `site overlap' is the usual overlap defined in \eqref{r12}. The `bond overlap' between two states $\bos^1$ and $\bos^2$, on the other hand, is defined as
\begin{equation}\label{q12}
Q_{1,2} := \frac{1}{|E|} \sum_{(i,j)\in E} \sigma_i^1 \sigma_j^1\sigma_i^2\sigma_j^2.
\end{equation}
We show that the bond overlap in the E-A model is not chaotic with respect to small fluctuations of the couplings at any temperature. This does not say anything about the site overlap; the site overlap in the E-A model can well be chaotic with respect to small fluctuations of the couplings, as predicted in~\cite{fisherhuse86, braymoore87}.
\begin{thm}\label{nochaos}
Suppose the E-A Hamiltonian \eqref{eahamil} on a graph $G$ is continuously perturbed up to time $t\ge 0$, according to the definition of continuous perturbation in Section~\ref{chaossec}. Let $\bos^1$ be chosen from the original Gibbs measure at inverse temperature $\beta$ and $\bos^2$ is chosen from the perturbed measure. Let the bond overlap $Q_{1,2}$ be defined as in~\eqref{q12}. Let
\[
{\textstyle q := \min \{\beta^2, \frac{1}{4d^2}\}},
\]
where $d$ is the maximum degree of $G$.
Then
\[
\ee(Q_{1,2}) \ge Cqe^{-t/Cq},
\]
where $C$ is a positive absolute constant. Moreover, the result holds for $\beta=\infty$ also, with the interpretation that the Gibbs measure at $\beta = \infty$ is just the uniform distribution on the set of ground states. 
\end{thm}
An interesting case of the above theorem is when $t = 0$. The result then says that if two configurations are drawn independently from the Gibbs measure, they have a non-negligible bond overlap with non-vanishing probability. The fact that this holds at any finite temperature is in contrast with the mean-field case (i.e.\ the S-K model), where there is a high-temperature phase ($\beta < 1$) where the bond overlap becomes negligible.
%[Another aside for experts on spin glasses: A particularly interesting case of the above theorem is when $t = 0$. The result then says that if two configurations are drawn independently from the Gibbs measure, they have a non-negligible bond overlap with non-vanishing probability. The fact that this holds at any finite temperature is in contrast with the mean-field case (i.e.\ the S-K model), where there is a high-temperature phase ($\beta < 1$) where the bond overlap (which is just the square of the site overlap in the S-K model) becomes negligible with high probability. This provides some credence to the folklorish claim that the Edwards-Anderson model does not, in fact, have a high temperature phase. A private discussion with Daniel Fisher also led to the conclusion that this comes very close to showing that there cannot be `infinitely many incongruent states' in the Edwards-Anderson model --- an important physics conjecture from \cite{fisherhuse87} that the author does not fully understand.]

However, while Theorem \ref{nochaos} establishes that the bond overlap does not become zero for any amount of perturbation, it does exhibit a sort of `quenched chaos', in the following sense.
\begin{thm}\label{quenchedchaos}
Fix $t > 0$ and let $Q_{1,2}$ be as in Theorem \ref{nochaos}. Then
\[
\ee\av{(Q_{1,2} - \av{Q_{1,2}})^2} \le \frac{2}{\beta e^{-{t/2}} \sqrt{t|E|}}. 
\]
\end{thm}
That is, if we perturb the system by an amount $t\gg |E|^{-1}$, the bond overlap between two configurations drawn from the two Gibbs measures is approximately equal to the quenched average of the overlap. In physical terms, the overlap `self-averages'. 

The combination of the last two theorems brings to light a surprising phenomenon. On the one hand, the perturbation retains a memory of the original Gibbs measure, because the overlap is non-vanishing in Theorem~\ref{nochaos}. On the other hand, the perturbation causes a chaotic reorganization of the Gibbs measure in such a way that the overlap concentrates on a single value in Theorem \ref{quenchedchaos}. The author can see no clear explanation of this confusing outcome. 

\subsection{Absence of superconcentration in the E-A model}
The proof of Theorem \ref{nochaos} is based on the following result, which says that the free energy is not superconcentrated in the E-A model on bounded degree graphs. This generalizes a well-known result of Wehr and Aizenman \cite{aizenmanwehr90}, who proved the analogous result on square lattices. The relative advantage of our approach is that it does not use the structure of the graph, whereas the Wehr-Aizenman proof depends heavily on properties of the lattice. 
\begin{thm}\label{nosuper}
Let $F(\beta)$ denote the free energy in the Edwards-Anderson model on a graph $G$, defined in \eqref{free}. Let $d$ be the maximum degree of $G$. Then for any $\beta$, including $\beta = \infty$ (where the free energy is just the energy of the ground state), we have
\[
\var F(\beta) \ge \frac{9|E|}{32}\min\biggl\{\beta^2, \frac{1}{4d^2}\biggr\}. 
\]
\end{thm}
The above result is based on a formula (Theorem \ref{varformula}) for the variance of an arbitrary smooth function of Gaussian random variables.

\subsection{A note about other models}
It will clear from our proofs that the chaos and superconcentration results hold for the $p$-spin versions of the S-K model for even $p$. (See Chapter 6 of \cite{talagrand03} for the definition of these models and various results.) In  fact, a generalization of Theorem \ref{chaoscont} is proven in Theorem \ref{pspin} later, which includes the $p$-spin models for even $p$. 

It will also be clear that the lack of superconcentration is true in the Random Field Ising Model on general bounded degree graphs. (Again, the lattice case is handled in \cite{aizenmanwehr90}. We refer to \cite{aizenmanwehr90} for the definition of the RFIM.) The absence of superconcentration in the RFIM implies that the {\it site} overlap is stable under perturbations, instead of the bond overlap as in the E-A model.

A simple model where our techniques give sharp results is the Random Energy Model (REM). This is discussed in Subsection \ref{rem}. 

\subsection{Unsolved questions}
In spite of the progress made in this paper over~\cite{chatterjee08}, many key issues are still out of reach. Some of them are as follows:
\begin{enumerate}
\item Improve the multiple valley theorem (Theorem \ref{multisk}) so that $\delta_N$ is a negative power of $N$, preferably better than $N^{-1/2}$, which will prove `strong multiple valleys' in the sense of \cite{chatterjee08}. 
\item Another possible improvement to Theorem \ref{multisk} can be achieved by increasing $r_N$ to something of the form $\exp (N^\alpha)$. 
\item Prove the chaos theorems (Theorems~\ref{chaosdisc} and \ref{chaoscont}) for the ground state ($\beta = \infty$) of the S-K model.
\item Improve the superconcentration result (Theorem \ref{superconc}) so that the right hand side is $N^\alpha$ for some $\alpha < 1$. This is tied to the improvement of the chaos result. 
\item If the above is not possible, at least prove a version of the superconcentration result where the right hand side does not depend on~$\beta$, or has a better dependence than $\log \beta$. This will solve the question of chaos for $\beta = \infty$. 
\item Prove that the site overlap in the Edwards-Anderson model is chaotic with respect to fluctuations in the disorder, even though the bond overlap is not.
\item Prove disorder chaos in the S-K model with nonzero external field, that is, if there is an additional term of the form $h \sum \sigma_i$ in the Hamiltonian. The general nature of the S-K model indicates that any result for $h\ne 0$ may be substantially harder to prove than for~$h=0$. (Reportedly, a sketch of the proof in this case will appear in the new edition of \cite{talagrand03}.)
\item Show that in the E-A model, the variance of $\av{Q_{1,2}}$ tends to zero and the graph size goes to infinity.
\item Establish temperature chaos in any of these models. 
\end{enumerate}
The rest of the paper is organized as follows. In Section \ref{sketches}, we sketch the proofs of the main results. In Section \ref{proofs}, we present some general results that cover a wider class of Gaussian fields. All proofs are given in Section~\ref{proofs}. 

\section{Proof sketches}\label{sketches}
In this section we give very short sketches of some of the main ideas of this paper. 
\subsection{Multiple valleys from chaos}
Suppose we choose $\bos^1$ from the Gibbs measure $G_N$ at inverse temperature $\beta$ and $\bos^2$ from the  measure $G_N'$ obtained by applying a continuous perturbation up to time $t$. Let $H_N$ and $H_N'$ be the two Hamiltonians. Suppose $\beta = \beta(N) \ra \infty$ and $t = t(N)\ra 0$ sufficiently slowly so that chaos holds (i.e.\ $\ee(R_{1,2}^2) \ra 0$ as $N \ra \infty$). Clearly this is possible by Theorem \ref{chaoscont}. Then due to chaos, $\bos^1$ and $\bos^2$ are approximately orthogonal. Since $\beta \ra \infty$, $\bos^1$ nearly minimizes $H_N$ and $\bos^2$ nearly minimizes $H_N'$.
But, since $t \ra 0$, $H_N \approx H_N'$.
Thus, $\bos^1$ and $\bos^2$ both nearly minimize $H_N$. This procedure finds two states that have nearly minimal energy and are nearly orthogonal. Repeating this procedure, we find many such states. The details are of this argument are worked out in Subsection \ref{multval}.

\subsection{Superconcentration iff chaos under continuous perturbations}
Let $\phi(t)$ denote $\ee(R_{1,2}^2)$ when $\bos^1$ is drawn from the unperturbed Gibbs measure at inverse temperature $\beta$ and $\bos^2$ is drawn from the Gibbs measure continuously perturbed up to time $t$. Let $F_N(\beta)$ be the free energy defined in \eqref{free}. Then we show that
\begin{equation}\label{varsketch}
\var(F_N(\beta)) = N\int_0^\infty e^{-t}\phi(t) dt. 
\end{equation}
The proof of this result (Theorem \ref{supergauss}) is simply a combination of the heat equation for the Ornstein-Uhlenbeck process and integration-by-parts. The formula directly shows that $\var(F_N(\beta)) = o(N)$ whenever $\phi(t)$ falls of sharply to zero, which is a way of saying that chaos implies superconcentration.

In Subsection \ref{chaosgauss}, we show that $\phi$ is a nonnegative and decreasing function. This proves the converse implication, since the integral of a nonnegative decreasing function can be small only if the function drops off sharply to zero.

\subsection{Chaos under continuous perturbations} 
Suppose $\bos^1$ is drawn from the Gibbs measure of the S-K model at inverse temperature $\beta$, and $\bos^2$ from the measure continuously perturbed up to time $t$. Let $R_{1,2}$ be the overlap of $\bos^1$ and $\bos^2$, as usual, and let 
\[
\phi_k(t) := \ee(R_{1,2}^{2k}).
\] 
%\begin{align*}
%\phi_k(t) &:= \ee(R_{1,2}^{2k})%\\
%&= \ee\left(\frac{\sum_{\bos^1,\bos^2}\left(\frac{\sum\sigma_i^1\sigma_i^2}{N}\right)^{2k}e^{\frac{\beta}{\sqrt{N}} \sum_{i<j} (g_{ij}\sigma_i^1\sigma_j^1 + g_{ij}^t\sigma_i^2\sigma_j^2)}}{\sum_{\bos^1,\bos^2}e^{\frac{\beta}{\sqrt{N}} \sum_{i<j} (g_{ij}\sigma_i^1\sigma_j^1 + g_{ij}^t\sigma_i^2\sigma_j^2)}}\right),
%\end{align*}
%where $g_{ij}^t := \est g_{ij} + \esst g_{ij}'$.
We have to show that for all $t$, 
\[
\phi_k(t) \le C N^{-k\min\{1,t/C\}}
\]
where $C$ is some constant that depends only on $\beta$. 

By repeated applications of differentiation and  Gaussian integration-by-parts, we show that  $(-1)^j \phi^{(j)}_k(t) \ge 0$ for all $t$ and $j$. Here $\phi^{(j)}_k$ denotes the $j$th derivative of $\phi_k$. Such functions are called  completely monotone.
Now, by a classical theorem of Bernstein about completely monotone functions, there is a probability measure $\mu_k$ on $[0,\infty)$ such that
\begin{equation}\label{inter}
\phi_k(t) = \phi_k(0)\int_0^\infty e^{-xt} d\mu_k(x). 
\end{equation}
By H\"older's inequality and the above representation, it follows that for $0\le t< s$,
\[
\phi_k(t) \le \phi_k(0)^{1-t/s} \phi_k(s)^{t/s}. 
\]
In other words, {\it chaos under large perturbations implies chaos under small perturbations}.
 Thus, it suffices to prove that $\phi_k(s)\le const. N^{-k}$ for sufficiently large $s$.

The next step is an `induction from infinity'. It is not difficult to see that when $t = \infty$, after integrating out the disorder,  $\bos^1$ and $\bos^2$ are independent and uniformly distributed on $\{-1,1\}^N$. From this it follows that $\phi_k(\infty) = const.N^{-k}$. We use this to obtain a similar bound on $\phi_k(s)$ for sufficiently large $s$, through the following steps. First, we show that for any $k$ and $s$, 
\[
\phi_k'(s) \ge -2N\beta^2 e^{-s} \phi_{k+1}(s).
\]
Thus, we have a {\it chain of differential inequalities}. 
It is possible to manipulate this chain to conclude that
\[
\phi_k(s) \le 2^{-2N}\sum_{\bos^1,\bos^2}\biggl(\frac{\bos^1\cdot\bos^2}{N}\biggr)^{2k} \exp\biggl(2\beta^2 e^{-s}\frac{(\bos^1\cdot\bos^2)^2}{N}\biggr). 
\]
The right hand side is bounded by $const. N^{-k}$ if and only if $s$ is sufficiently large. (This is related to the fact that when $Z$ is a standard Gaussian random variable, $\ee(e^{\alpha Z^2}) < \infty$ if and only if $\alpha < 1/2$.) This completes the proof sketch. The details of the above argument are worked out in Subsection \ref{chaosgauss}. 

\subsection{Chaos in E-A model}
The proof of Theorem \ref{nochaos}, again, is based on the representation~\eqref{varsketch} of the variance of the free energy and the representation \eqref{inter} of the function $\phi$ (both of which hold for the E-A model as well). From  \eqref{inter}, it follows that there is a nonnegative random variable $U$ such that for all $t\ge 0$,
\[
\phi(t) = \phi(0)\ee(e^{-tU}).  
\]
From this and \eqref{varsketch} it follows that 
\[
\var F(\beta) = N\phi(0) \ee((1+U)^{-1}). 
\]
Next, we prove a simple analytical fact: Suppose $V$ is a nonnegative random variable and let $v := \ee((1+V)^{-1})$.  Then for any $t\ge 0$, 
\[
\ee(e^{-tV}) \ge \frac{1}{2}ve^{-t(2-v)/v}.
\] 
Using this inequality for the random variable $U$ and  the lower bound on the variance from Theorem \ref{nosuper}, it is easy to obtain the required lower bound on the function $\phi(t)$, which establishes the absence of chaos. The details of this argument are presented in Subsection \ref{nochaosproof}. 

The proof of Theorem \ref{quenchedchaos} involves a new idea. Let $\bg = (g_{ij})_{(i,j)\in E}$, and let $\bg', \bg''$ be independent copies of $\bg$. For each $t$, let 
\[
\bg^t := \est \bg + \esst \bg', \ \ \ \bg^{-t} := \est \bg + \esst \bg''.
\]
For each $t\in \rr$, let $\bos^t$ denote a configuration drawn from the Gibbs measure defined by the disorder $\bg^t$. For $t\ne s$, we assume that $\bos^t$ and $\bos^s$ are independent given $\bg, \bg', \bg''$. Define
\[
\phi(t) := \frac{1}{|E|} \sum_{(i,j)\in E}\ee\bigl(\av{\sigma_i^t\sigma_j^t}\av{\sigma_i^{-t}\sigma_j^{-t}}). 
\]
By a similar logic as in the derivation of \eqref{inter}, one can show that $\phi$ is a completely monotone function. Also, $\phi$ is bounded by $1$. Thus, for any $t > 0$,
\begin{equation}\label{pbd0}
|\phi'(t)| \le \frac{\phi(0)-\phi(t)}{t} \le \frac{1}{t}. 
\end{equation}
Now fix $t$ and let
\[
\bar{u}_{ijkl} := \ee(\av{\sigma_i^t \sigma_j^t \sigma_k^t \sigma_l^t}\mid \bg), \ \ \ \bar{v}_{ijkl} := \ee(\av{\sigma_i^t \sigma_j^t}\av{ \sigma_k^t \sigma_l^t}\mid \bg). 
\]
It turns out that
\[
\phi'(t) = -\frac{2e^{-2t} \beta^2}{|E|} \sum_{(i,j)\in E, \, (k,l)\in E} \ee((\bar{u}_{ijkl} - \bar{v}_{ijkl})^2)
\]
and 
\[
\ee\av{(Q_{\bos^t, \bos^{-t}}- \av{Q_{\bos^t, \bos^{-t}}})^2} = \frac{1}{|E|^2}\sum_{(i,j)\in E,\, (k,l)\in E} \ee(\bar{u}_{ijkl}^2 - \bar{v}_{ijkl}^2),
\]
where $Q_{\bos^t, \bos^{-t}}$ is the bond overlap between $\bos^t$ and $\bos^{-t}$. Combining these two identities with the inequality \eqref{pbd0}, it is easy to complete the proof of Theorem \ref{quenchedchaos}. The details are in Subsection \ref{quenchedproof}. 

\subsection{Chaos under discrete perturbations}
Let $\bg = (g_{ij})_{1\le i,j\le N}$, and let $\bg'$ be an independent copy of $\bg$. For any $A\subseteq \{(i,j): 1\le i,j\le n\}$, let $\bg^A$ be the array whose $(i,j)$th component is
\[
g_{ij}^A := 
\begin{cases}
g_{ij}' &\text{ if } (i,j)\in A,\\
g_{ij} &\text{ if } (i,j) \not\in A.
\end{cases}
\]
Let $F_N$ be the free energy, considered as a function of $\bg$. Suppose $\epsilon_N$ and $\delta_N$ are constants such that for all $i,j$, 
\[
\biggl|\fpar{F_N}{g_{ij}}\biggr|\le N^{1/2}\delta_N \ \ \text{ and } \ \ \biggl|\spar{F_N}{g_{ij}}\biggr|\le N^{1/2}\epsilon_N \ \ \text{almost surely.}
\]
Fix $0\le k\le N^2$, and let $A$ be a subset of~$\{(i,j): 1\le i,j\le N\}$, chosen uniformly at random from the collection of all subsets of size~$k$. Let $\bos^1$ be chosen from the Gibbs measure at inverse temperature $\beta$ defined by the disorder $\bg$, and let $\bos^2$ be drawn from the Gibbs measure defined by $\bg^A$. Let $R_{1,2}$ denote the overlap of $\bos^1$ and $\bos^2$, as usual. The key step is to prove that  for some absolute constant $C$, 
\begin{equation*}
\ee(R_{1,2}^2) \le \frac{CN}{k}\var(F_N) + CN^2\delta_N\epsilon_N. 
\end{equation*}
This inequality is the content of Theorem \ref{discgauss}. 
The proof is completed by showing that we can choose $\delta_N$ and $\epsilon_N$ such that $\delta_N \epsilon_N = o(N^{-2})$, and using the superconcentration bound (Theorem \ref{superconc}) on the variance of $F_N$. The details of the proof are given in Subsection \ref{chaosdiscproof}. 

\subsection{No superconcentration in the E-A model}
Although this result was already proven in \cite{aizenmanwehr90} for the E-A model on lattices, it may be worth sketching our argument for general bounded degree graphs here. Our proof is based on a general lower bound for arbitrary functions of Gaussian random variables. The result (Theorem \ref{varlowbd}) goes as follows: Suppose $f:\rr^n \ra \rr$ is an absolutely continuous function such that there is a version of its gradient $\nabla f$ that is bounded on bounded sets. Let $\bg$ be a standard Gaussian random vector in $\rr^n$, and suppose $\ee|f(\bg)|^2$ and $\ee|\nabla f(\bg)|^2$ are both finite. Then
\[
\var(f(\bg)) \ge \frac{1}{2}\sum_{i=1}^n \biggl(\ee\biggl(g_i \fpar{f}{g_i}\biggr)\biggr)^2 \ge \frac{1}{2n}\bigl(\ee(\bg \cdot \nabla f(\bg))\bigr)^2,
\]
where $\bx \cdot \by$ denotes the usual inner product on $\rr^n$. We apply this result to the Gaussian vector $\bg = (g_{ij})_{(i,j)\in E}$, taking the function $f(\bg)$ to be the free energy $F(\beta)$. A few tricks are required to get a lower bound on the right hand side that does not blow up as $\beta \ra \infty$. 

Incidentally, the above lower bound on the variance of Gaussian functionals is based on a multidimensional Plancherel formula that may be of independent interest:
\begin{equation}\label{planform}
\var(f(\bg)) = \sum_{k=1}^\infty \frac{1}{k!}\sum_{1\le i_1,\ldots,i_k\le n} \biggl(\ee\biggl(\frac{\partial^k f}{\partial g_{i_1}\cdots \partial g_{i_k}}\biggr)\biggr)^2. 
\end{equation}
Versions of this formula have been previously derived in the literature using expansions with respect to the multivariate orthogonal Hermite polynomial basis (see Subsection \ref{plancherel} for references). We give a different  proof avoiding the use of the orthogonal basis.

\section{General results about Gaussian fields and proofs}\label{proofs}
The results of Section \ref{intro} are applications of some general theorems about Gaussian fields. These are presented in this section, together with the proofs of the theorems of Section \ref{intro}. Unlike the previous sections, we proceed according to the theorem-proof format in the rest of the paper. %The proofs of the results of Section \ref{intro} will be given at appropriate locations along the way. 

\subsection{Chaos in Gaussian fields}\label{chaosgauss}
Let $S$ be a finite set and let $\bbx = (X_i)_{i\in S}$ be a centered Gaussian random vector. Let 
\[
\rho_{ij} := \cov(X_i, X_j).
\]
Let $\bbx'$ be an independent copy of $\bbx$, and for each $t\ge 0$, let
\[
\bbx^t := \est \bbx + \esst \bbx'. 
\]
Fix $\beta \ge 0$. For each $t,s\ge 0$, define a probability measure $G_{t,s}$ on $S\times S$ that assigns mass
\[
\frac{e^{\beta X_i^t + \beta X_j^s}}{\sum_{k,l}e^{\beta X_k^t + \beta X_l^s}}
\]
to the point $(i,j)$, for each $(i,j)\in S\times S$. The average of a function $h:S\times S \ra \rr$ under the measure $G_{t,s}$ will be denoted by $\smallavg{h}_{t,s}$, that is, 
\[
\av{h}_{t,s} := \frac{\sum_{i,j} h(i,j) e^{\beta X_i^t +\beta  X_j^s}}{\sum_{i,j}e^{\beta X_i^t + \beta X_j^s}}.
\]
We will consider the covariance kernel $\rho$ as a function on $S\times S$, defined as $\rho(i,j) := \rho_{ij}$. Alternatively, it will also be considered as a square matrix. %It should not be confused with the overlap \eqref{r12} defined in  Section \ref{intro}. 
\begin{thm}\label{main}
Assume that $\rho_{ij}\ge 0$ for all $i,j$. For each $i$, let
\[
\nu_i := \ee\biggl(\frac{e^{\beta X_i}}{\sum_j e^{\beta X_j}}\biggr). 
\]
Let $\phi(x) = \sum_{k=0}^\infty c_k x^k$ be any convergent power series on $[0,\infty)$ all of whose coefficients are nonnegative. Then for each $t\ge 0$,
\[
\ee\av{\phi\circ \rho}_{0,t} \le \inf_{s\ge t} \bigl(\ee\av{\phi\circ \rho}_{0,0}\bigr)^{1-t/s} \biggl(\sum_{i,j} \phi(\rho_{ij}) e^{2\beta^2 e^{-s}\rho_{ij}}\nu_i \nu_j\biggr)^{t/s}. 
\]
Moreover, $\ee\av{\phi\circ \rho}_{0,t}$ is a decreasing function of $t$. 
\end{thm}
Roughly, the way to apply this theorem is the following: prove that the right hand side is small for some large $t$ using high temperature methods, and then use the infimum to show that the smallness persists for small $t$ as well.

Since the application of Theorem \ref{main} to the S-K model seems to yield a suboptimal result (Theorem \ref{chaoscont}), one can question whether Theorem \ref{main} can ever give sharp bounds. In Subsection~\ref{rem} we settle this issue by  showing that Theorem \ref{main} gives a  sharp result for Derrida's Random Energy Model. 

Let us now proceed to prove Theorem \ref{main}. In the following, $C^\infty_b(\rr^S)$ will denote the set of all infinitely differentiable real-valued functions on $\rr^S$ with bounded derivatives of all orders.

Let us first extend the definition of $\bbx^t$ to negative $t$. This is done quite simply. Let $\bbx''$ be another independent copy of $\bbx$ that is also independent of $\bbx'$, and for each $t\ge 0$, let
\[
\bbx^{-t} := \est \bbx + \esst \bbx''.
\]
Let us now recall Gaussian integration by parts: If $f:\rr^S \ra \rr$ is an absolutely continuous function such that $|\nabla f(\bbx)|$ has finite expectation, then for any $i\in S$,
\[
\ee(X_i f(\bbx)) = \sum_{j\in S} \rho_{ij} \ee(\partial_j f(\bbx)),
\]
where $\partial_j f$ denotes the partial derivative of $f$ along the $j$th coordinate (see e.g.\ \cite{talagrand03}, Appendix A.6). The following lemma is simply a reformulated version of the above identity. 
\begin{lmm}\label{basic}
For any $f\in C^\infty_b(\rr^S)$, we have
\begin{equation*}
\frac{d}{dt}\ee(f(\bbx^{-t})f(\bbx^t)) = -2e^{-2t} \sum_{i,j} \rho_{ij} \ee\bigl(\partial_i f(\bbx^{-t}) \partial_j f(\bbx^t)\bigr).
\end{equation*}
\end{lmm}
\begin{proof}
For each $t\ge 0$, define 
\[
\bby^t := \esst \bbx - \est \bbx', \ \ \ \bby^{-t} := \esst \bbx - \est \bbx''. 
\]
A simple computation gives
\begin{align*}
&\frac{d}{dt} \ee(f(\bbx^{-t}) f(\bbx^t)) \\
&= -\frac{\est}{\esst}\ee\bigl((\bby^{-t}\cdot \nabla f(\bbx^{-t}))f(\bbx^t) + (\bby^{t}\cdot \nabla f(\bbx^{t})) f(\bbx^{-t})\bigr)\\
&= -\frac{2\est}{\esst} \ee((\bby^{-t}\cdot \nabla f(\bbx^{-t})) f(\bbx^{t})).
\end{align*}
(Note that issues like moving derivatives inside expectations are easily taken care of due to the assumption that $f\in C^\infty_b$.) 
One can verify by computing covariances that $\bby^{-t}$ and the pair $(\bbx^{-t}, \bbx')$ are independent. Moreover,
\[
\bbx^{t} = e^{-2t} \bbx^{-t} + \est\esst \bby^{-t} + \esst \bbx'. 
\]
So for any $i$, Gaussian integration by parts gives 
\begin{align*}
\ee\bigl(Y_i^{-t} \partial_i f(\bbx^{-t}) f(\bbx^{t})\bigr) &= \est\esst \sum_j \rho_{ij} \ee\bigl(\partial_i f(\bbx^{-t}) \partial_j f(\bbx^{t})\bigr).
\end{align*}
The proof is completed by combining the last two steps.
\end{proof}
Our next lemma is the most crucial component of the whole argument. It gives a way of extrapolating high temperature results to the low temperature regime.  
\begin{lmm}\label{complete}
Let $\mf$ be the class of all functions $h$ on $[0,\infty)$ that can be expressed as
\[
h(t) = \sum_{i=1}^m e^{-c_i t} \ee(f_i(\bbx^{-t}) f_i(\bbx^t))
\]
for some nonnegative integer $m$ and nonnegative real numbers $c_1, c_2,\ldots, c_m$, and functions $f_1,\ldots,f_m$ in $C^\infty_b(\rr^S)$.  For any $h\in \mf$,  there is a probability measure $\mu$ on $[0,\infty)$ such that for each $t\ge 0$,
\[
h(t) = h(0)\int_{[0,\infty)} e^{-xt} d\mu(x). 
\]
In particular, for any $0< t\le s$,
\[
h(t)\le h(0)^{1-t/s} h(s)^{t/s}.
\]
\end{lmm}
\begin{proof}
Note that any $h\in \mf$ must necessarily be a nonnegative function, since $\bbx^{-t}$ and $\bbx^t$ are independent and identically distributed conditional on $\bbx$, which gives
\[
\ee(f(\bbx^{-t})f(\bbx^t)) = \ee\bigl((\ee(f(\bbx^t)\mid \bbx))^2\bigr). 
\]
Now, if $h(0)=0$, then $h(t)=0$ for all $t$, and there is nothing to prove. So let us assume $h(0) > 0$.

Since $\rho$ is a positive semidefinite matrix, there is a square matrix $C$ such that $\rho = C^T C$. Thus, given a function $f$, if we define
\[
g_i := \sum_j C_{ij} \partial_j f,
\]
then by Lemma \ref{basic} we have
\[
\frac{d}{dt} \ee(f(\bbx^{-t}) f(\bbx^t)) = -2e^{-2t}\sum_i \ee(g_i(\bbx^{-t}) g_i(\bbx^t)). 
\]
From this observation and the definition of $\mf$, it follows easily that if $h\in \mf$, then $-h' \in \mf$. Proceeding by induction, we see that for any $k$, $(-1)^k h^{(k)}$ is a nonnegative function (where $h^{(k)}$ denotes the $k$th derivative of $h$). Such functions on $[0,\infty)$  are called `completely monotone'. The most important property of completely monotone functions (see e.g.\ Feller \cite{feller71}, Vol. II, Section XIII.4) is that any such function $h$ can be represented as the Laplace transform of a positive Borel measure $\mu$ on $[0,\infty)$, that is, 
\[
h(t) = \int_{[0,\infty)} e^{-x t} d\nu(x). 
\]
Moreover, $h(0)= \nu(\rr)$. By taking $\mu(dx) = h(0)^{-1}\nu(dx)$, this proves the first assertion of the theorem. For the second, note that by H\"older's inequality, we have that for any $0< t\le s$,
\[
\int_{\rr} e^{-xt} d\mu(x) \le \biggl(\int_{\rr} e^{-xs}d\mu(x)\biggr)^{t/s} = (h(s)/h(0))^{t/s}.
\]
This completes the proof. 
\end{proof}
The next lemma is obtained by a variant of the Gaussian interpolation methods for analyzing mean field spin glasses at high temperatures. It is similar to R.\ Lata\l a's unpublished proof of the replica symmetric solution of the S-K model (to appear in the new edition of \cite{talagrand03}). %The following result proves Theorem \ref{main} without the infimum. 
\begin{lmm}\label{hightemp}
Let $\phi$ and $\nu_i$ be as in Theorem \ref{main}. Then for each $t\ge 0$,
\[
\ee\av{\phi\circ \rho}_{-t,t} \le \sum_{i,j} \phi(\rho_{ij}) e^{2\beta^2 e^{-2t}\rho_{ij}}\nu_i \nu_j. 
\]
\end{lmm}
\begin{proof}
For each $i$, define a function $p_i:\rr^S \ra \rr$ as 
\[
p_i(\bx) := \frac{e^{\beta x_i}}{\sum_j e^{\beta x_j}}. 
\]
Note that
\[
\partial_j p_i = \beta(p_i \delta_{ij} - p_ip_j),
\]
where $\delta_{ij} = 1$ if $i = j$ and $0$ otherwise. Since $p_i$ is bounded, this proves in particular that $p_i\in C^\infty_b (\rr^S)$.

Take any nonnegative integer $r$. Since $\rho = (\rho_{ij})$ is a positive semidefinite matrix, so is $\rho^{(r)} := (\rho_{ij}^r)$. (To see this, just note that $\bbx^1,\ldots, \bbx^r$ are independent copies of $\bbx$, then $\cov(X_i^1\cdots X_i^r,  X_j^1\cdots X_j^r) = \rho_{ij}^r$.) 
Therefore there exists a matrix $C^{(r)} = (C^{(r)}_{ij})$ such that $\rho^{(r)} = (C^{(r)})^T C^{(r)}$. Define the functions 
\[
h_i := \sum_j C_{ij}^{(r)} p_j, \ \ \ i\in S. 
\]
In the following we will denote $p_i(\bbx^s)$ and $h_i(\bbx^s)$ by $p_i^s$ and $h_i^s$ respectively, for all $s\in \rr$. Let 
\begin{align*}
f_r(t) &:= \ee\biggl(\sum_{i,j} \rho_{ij}^r p_i^{-t} p_j^t \biggr) = \ee\biggl(\sum_i h_i^{-t} h_i^t\biggr).
\end{align*}
By Lemma \ref{basic} we get
\begin{align*}
f_r'(t) &= -2e^{-2t}\sum_{i} \sum_{k,l} \rho_{kl}\ee\bigl(\partial_k h_i^{-t} \partial_l h_i^t\bigr)\\
&= -2\beta^2e^{-2t} \sum_{i,k,l} \rho_{kl} \ee\biggl(\biggl(C_{ik}^{(r)}p_k^{-t} - \sum_{j} C_{ij}^{(r)}p_j^{-t} p_k^{-t}\biggr)\\
&\hskip2in \times \biggl(C_{il}^{(r)}p_l^{t} - \sum_{j} C_{ij}^{(r)}p_j^{t} p_l^{t}\biggr) \biggr)\\
&= -2\beta^2e^{-2t} \ee\biggl(\sum_{k,l}\rho_{kl}^{r+1} p_k^{-t} p_l^{t} -  \sum_{j,k,l} \rho_{kl} \rho_{jl}^r p_j^{-t} p_k^{-t} p_l^t \\
&\hskip1in - \sum_{j,k,l} \rho_{kl} \rho_{kj}^r p_k^{-t} p_j^t p_l^t + \sum_{j,k,l,m} \rho_{kl} \rho_{jm}^r p_j^{-t} p_k^{-t} p_m^t p_l^t\biggr). 
\end{align*}
Our objective is to get a lower bound for $f_r'(t)$. For this, we can delete the two middle terms in the above expression because they contribute a positive amount. For the fourth term, note that by H\"older's inequality,
\begin{align*}
&\biggl(\sum_{k,l}\rho_{kl}p_k^{-t} p_l^t\biggr)\biggl(\sum_{j,m} \rho_{jm}^r p_j^{-t} p_m^t\biggr) \\
&\le \biggl(\sum_{k,l}\rho_{kl}^{r+1}p_k^{-t} p_l^t\biggr)^{\frac{1}{r+1}} \biggl(\sum_{j,m} \rho_{jm}^{r+1} p_j^{-t} p_m^t\biggr)^{\frac{r}{r+1}} = \sum_{k,l} \rho_{kl}^{r+1} p_k^{-t} p_l^t.
\end{align*}
Thus, by Lemma \ref{complete} and the above inequalities, we have 
\begin{equation}\label{diffineq1}
0\ge f_r'(t) \ge -4\beta^2 e^{-2t} f_{r+1}(t).
\end{equation}
Now let $g_r(u) := f_r(-\log \sqrt{u})$ for $0 < u< 1$.  Then
\[
g_r'(u) = -\frac{f_r'(-\log \sqrt{u}) }{2u}. 
\]
The inequality \eqref{diffineq1} simply becomes
\begin{equation}\label{diffineq2}
0\le g_r'(u) \le 2\beta^2g_{r+1}(u).
\end{equation}
Fix $0< u< 1$, $r\ge 1$. For each $m \ge 1$, let
\begin{align*}
T_m &:= \int_0^u\int_0^{u_1} \cdots \int_0^{u_{m-1}}
(2\beta^2)^{m-1}g'_{r+m-1}(u_m)du_m du_{m-1} \cdots du_1.
\end{align*}
Using (\ref{diffineq2}), we see that 
\begin{align*}
&0\le T_m \le \int_0^u\int_0^{u_1} \cdots \int_0^{u_{m-1}}
(2\beta^2)^m g_{r+m}(u_m)  du_m du_{m-1} \cdots du_1  \\
&= \int_0^u \cdots \int_0^{u_{m-1}}
(2\beta^2)^m \biggl(g_{r+m}(0) +  \int_0^{u_m}
g'_{r+m}(u_{m+1})du_{m+1} \biggr) du_m  \cdots
du_1  \\ 
&= \frac{(2\beta^2)^m g_{r+m}(0)u^m}{m!} + T_{m+1}.
\end{align*}
Inductively, this implies that for any $m \ge 1$,
\[
g_r(u) = g_r(0) + T_1 \le \sum_{l=0}^m \frac{g_{r+l}(0) (2\beta^2 u)^l}{l!}
+ T_{m+1}.
\]
Again, for any $m \ge 2$, 
\begin{align*}
0\le T_m &\le \int_0^u\int_0^{u_1} \cdots \int_0^{u_{m-1}}
(2\beta^2)^m g_{r+m}(u_m)  du_m du_{m-1} \cdots du_1  \\
&\le \frac{M^{r+m} (2\beta^2u)^{m}}{m!}
\end{align*}
where $M = \max_{i,j} \rho_{ij}$.
Thus, $\lim_{m\ra \infty}T_m = 0$. Finally, observe that $p_i^\infty$ and $p_i^{-\infty}$ are independent. This implies that for any $m$,
\[
g_m(0) = f_m(\infty) = \sum_{i,j} \rho_{ij}^m\nu_i \nu_j.
\] 
Combining, we conclude that
\begin{equation*}\label{super}
\begin{split}
g_r(u) &\le \sum_{l=0}^\infty \frac{g_{r+l}(0)(2\beta^2u)^l}{l!} 
%&= \sum_{r=0}^\infty \frac{\bigl(\sum_{i,j}
%\sigma_{ij}^{r+1}n^{-2} \bigr)(4t)^r}{r!} \\
=\sum_{i,j} \rho_{ij}^r e^{2\beta^2 u\rho_{ij}}\nu_i \nu_j.
\end{split}
\end{equation*}
The result now follows easily by taking $u = e^{-2t}$ and summing over $r$, using the fact that $\phi$ has nonnegative coefficients in its power series. 
\end{proof}
\begin{proof}[Proof of Theorem \ref{main}]
Let $p_i^t$ be as in the proof of Lemma \ref{hightemp}. 
As noted in the proof of Lemma \ref{hightemp}, the matrix $(\rho_{ij}^r)$ is positive semidefinite for every nonnegative integer $r$. Since $\phi$ has nonnegative coefficients in its power series, it follows that the matrix $\Phi := (\phi(\rho_{ij}))$ is also positive semidefinite. Let $C=(C_{ij})$ be a matrix such that $\Phi = C^T C$. Then
\[
\av{\phi\circ \rho}_{-t,t} = \sum_{i,j} \phi(\rho_{ij}) p_i^{-t}p_j^t = \sum_i \biggl(\sum_j C_{ij} p_j^{-t}\biggr) \biggl(\sum_j C_{ij} p_j^{t}\biggr). 
\]
Therefore, the function
\[
h(t) :=\ee \av{\phi\circ \rho}_{-t,t}
\]
belongs to the class $\mf$ of Lemma \ref{complete}. The proof is now finished by using Lemma \ref{complete} and Lemma \ref{hightemp}, and the observation that $h(t/2) = \ee\av{\phi\circ \rho}_{0, t}$ (since $(\bbx^{-t/2}, \bbx^{t/2})$ has the same law as $(\bbx^0, \bbx^{t})$). The claim that $h(t)$ is a decreasing function of $t$ is automatic because $h'\le 0$. 
\end{proof}

\subsection{Proof of Theorem \ref{chaoscont}}
We are now ready to give a proof of Theorem~\ref{chaoscont} using Theorem~\ref{main}. In fact, we shall prove a slightly general result below, which also covers the case of $p$-spin models for even $p$, as well as further generalizations. 

Let $N$ be a positive integer and suppose $(H_N(\bos))_{\bos\in \{-1,1\}^N}$ is a centered Gaussian random vector with 
\[
\cov(H_N(\bos), H_N(\bos')) = N \xi(R_{\bos,\bos'}),
\]
where $\xi$ is some function on $[-1,1]$ that does not depend on $N$ and $R_{\bos,\bos'}$ is the overlap defined in \eqref{r12}. Let us fix $\beta \ge 0$. The Hamiltonian $H_N$ defines a Gibbs measure on $\{-1,1\}^N$ by putting mass proportional to $e^{-\beta H_N(\bos)}$ at each configuration~$\bos$. This class of models was considered by Talagrand \cite{talagrand06} in his proof of the generalized Parisi formula. For the S-K model, $\xi(x) = x^2/2$, while for the $p$-spin models, $\xi(x)=x^p/p!$. (We refer to Chapter 6 in \cite{talagrand03} for the definition of the $p$-spin models and related discussions.)

Let $H_N'$ be an independent copy of $H_N$, and for each $t\ge 0$, let
\[
H_N^t := \est H_N + \esst H_N'. 
\]
Given a function $h$ on $\{-1,1\}^N \times \{-1,1\}^N$, we define the average $\av{h(\bos, \bos')}_{t,s}$ as the average with respect to the product of the Gibbs measures defined by $H_N^t$ and $H_N^s$, that is, 
\[
\av{h(\bos, \bos')}_{t,s} := \frac{\sum_{\bos,\bos'} h(\bos, \bos') e^{-\beta H_N^t(\bos) - \beta H_N^s(\bos')}}{\sum_{\bos,\bos'}  e^{-\beta H_N^t(\bos) - \beta H_N^s(\bos')}}.
\]
The following result establishes the presence of chaos in this class of models under some restrictions on $\xi$. It is easy to see that the result covers all $p$-spin models for even $p$, and in particular, the original SK model. 
\begin{thm}\label{pspin}
Suppose $\xi$ is nonnegative, $\xi(1) = 1$ and there is a constant $c$ such that $\xi(x) \le c x^2$ for all $x\in [-1,1]$. Then there is a constant $C$ depending only on $c$ such that for all $t\ge 0$ and $\beta > 1$, and any positive integer $k$,
\[
\ee\av{(\xi(R_{\bos,\bos'}))^k}_{0,t} \le (Ck)^k N^{-k\min\{1, \;t/C\log(1+C \beta)\}}.
\]
\end{thm}
\begin{proof}
By symmetry, it is easy to see that for each $\bos$, 
\[
\ee\biggl(\frac{e^{-\beta H_N(\bos)}}{\sum_{\bos'} e^{-\beta H_N(\bos')}} \biggr) = 2^{-N}.
\]
Again, it follows from elementary  combinatorial arguments   that there are positive constants $\gamma$ and $C$ that do not depend on $N$, such that for any positive integer $k$ and any $N$,
\[
2^{-2N} \sum_{\bos, \bos'} \frac{(\bos\cdot \bos')^{2k}}{N^k} \exp\biggl(\frac{\gamma(\bos\cdot \bos')^2}{N}\biggr) \le (Ck)^k. 
\]
Choosing $s$ so that $2\beta^2c e^{-s}  = \min\{\gamma, 2\beta^2c\}$, and $\phi(x) = x^k/N^k$, we see from Theorem~\ref{main} (and the assumption that $\xi(x)\le cx^2$) that for any $0\le t\le s$
\begin{align*}
&\ee\av{(\xi(R_{\bos,\bos'}))^k}_{0,t} \\
&\le \bigl(\ee\av{(\xi(R_{\bos,\bos'}))^k}_{0,0}\bigr)^{1-t/s} \biggl(c^k2^{-2N}\sum_{\bos,\bos'} \frac{(\bos\cdot \bos')^{2k}}{N^{2k}} \exp\biggl(\frac{\gamma(\bos\cdot \bos')^2}{N}\biggr)\biggr)^{t/s} \\
&\le (Ck)^kN^{-kt/s},
\end{align*}
where $C$ is a constant that does not depend on $N$. This proves the result for $t\le s$. For $t\ge s$ we use the last assertion of Theorem \ref{main} to conclude that $\ee\av{(\xi(R_{\bos,\bos'}))^k}_{0,t}$ is decreasing in $t$. Finally, observe that 
\begin{align*}
s &= 
\begin{cases}
0 &\text{ if } 2\beta^2 c\le \gamma,\\
\log (2\beta^2c/\gamma) &\text{ if } 2\beta^2 c > \gamma
\end{cases}
\\
&\le C\log (1+C\beta) 
\end{align*}
for some constant $C$ that depends only on $c$ and $\gamma$.
\end{proof}
\subsection{Multiple valleys in Gaussian fields}\label{multval}
In this subsection we use Theorem \ref{main} to prove a multiple valley result for general Gaussian fields. Let all notation be the same as in Subsection \ref{chaosgauss}. The idea of the proof is borrowed from the proof of Theorem 3.7 in \cite{chatterjee08}, although there are added complications resulting from the fact that we are trying to derive a result about $\beta = \infty$ from a result about  finite $\beta$ (i.e.\ Theorem \ref{main}). 
\begin{thm}\label{multigauss}
Let $r$ be a positive integer, and let $\epsilon, \delta\in (0,1)$. Choose any $\beta > 0$ and $t> 0$. Let $M := \max_i X_i$, $m := \ee(M)$ and $\sigma^2 := \max_i \var(X_i)$. Define
\[ 
\gamma := \frac{8r\log |S|}{\delta m \beta} + \frac{2rt}{\delta} + e^{-m^2/8\sigma^2} +  \frac{r^2}{2\epsilon \sigma^2} \ee\av{\rho}_{-t,t} 
\]
where the Gibbs average in the last term is taken at inverse temperature~$\beta$. 
Then with probability at least $1-\gamma$, there exists a set $A\subseteq S$ of size $r$ such that for any distinct $i,j\in A$, we have $\rho_{ij}\le \epsilon \sigma^2$, and for all $i\in A$, $X_i\ge (1-\delta)M$. 
\end{thm}
\begin{proof}
Given $\bbx$, let $I$ be a random variable chosen from the set $S$ such that
\[
\pp(I=i \mid \bbx) = \frac{e^{\beta X_i}}{\sum_j e^{\beta X_j}}. 
\]
%We will denote the conditional expectation given $\bbx$ by $\smallavg{\cdot}$. 
Next, let $M := \max_i X_i$ and define a random function
\[
F(\beta) := \log \sum_i e^{\beta X_i}. 
\]
Then we have 
\begin{align*}
\beta M &= \log e^{\beta M} \le \log \sum_i e^{\beta X_i}\\
&\le \log (|S|  e^{\beta M}) = \log |S| + \beta M. 
\end{align*}
Thus, 
\[
|F(\beta) - \beta M| \le \log |S|. 
\]
An easy verification shows that
\[
F''(\beta) = \ee(X_I^2\mid \bbx) - (\ee(X_I\mid\bbx))^2\ge 0. 
\]
Therefore $F'$ is an increasing function of $\beta$ and hence 
\begin{align*}
F'(\beta) &\ge \frac{2}{\beta}\int_{\beta/2}^\beta F'(x) dx \\
&= \frac{F(\beta) - F(\beta/2)}{\beta/2} \ge M - \frac{4\log |S|}{\beta}. 
\end{align*}
Combining this with the observation that $F'(\beta ) = \ee(X_I\mid \bbx)$, we have 
\[
\ee(M-X_I) = \ee(M - F'(\beta))\le \frac{4\log |S|}{\beta}. 
\]
Now let $\bbz^{(1)}, \ldots, \bbz^{(r)}$ be i.i.d.\ copies of $\bbx$. Let
\[
\bbx^{(k)} := \est \bbx + \esst \bbz^{(k)},
\]
and 
\[
\bby^{(k)} := \esst \bbx - \est\bbz^{(k)}.
\]
Then $\bbx^{(k)}$ and $\bby^{(k)}$ are independent (jointly Gaussian and all covariances vanish), and
\[
\bbx = \est\bbx^{(k)} + \esst \bby^{(k)}. 
\]
Let $I^{(k)}$ be a random variable on $S$ whose conditional distribution given $\bbx^{(k)}$ is the same as that of $I$ given $\bbx$. In particular, by the independence of $\bbx^{(k)}$ and $\bby^{(k)}$, $I^{(k)}$ and $\bby^{(k)}$ are also independent. From this observation and the above representation of $\bbx$, we have 
\begin{align*}
\ee(X^{(k)}_{I^{(k)}}- X_{I^{(k)}})  &= (1-\est)\ee(X^{(k)}_{I^{(k)}}) - \esst \ee(Y^{(k)}_{I^{(k)}}) \\
&= (1-\est) \ee(X_I)\le t m. \ \ \text{(Recall: $m = \ee(M)$.)} 
\end{align*}
The last equality holds because $X_{I^{(k)}}^{(k)}$ and $X_I$ have the same unconditional distribution. For the same reason, we have 
\begin{align*}
\ee(M - X_{I^{(k)}}^{(k)}) &= \ee(M - X_I) \le \frac{4\log |S|}{\beta}. 
\end{align*}
Combining the last two inequalities, we see that
\[
\ee(M- X_{I^{(k)}}) \le \frac{4\log |S|}{\beta} + t m. 
\]
Thus,
\[
\ee\sum_{k=1}^r (M- X_{I^{(k)}}) \le  \frac{4r\log |S|}{\beta} + rt m.
\]
Now, we clearly have that for any $k\ne l$,
\[
\ee(\rho_{I^{(k)} I^{(l)}}) = \ee\av{\rho}_{-t,t}.
\]
Thus, 
\[
\ee\sum_{1\le k< l\le r} \rho_{I^{(k)} I^{(l)}} = \frac{r(r-1)}{2} \ee\av{\rho}_{-t,t}. 
\]
Finally, by Gaussian concentration (see Proposition 1.3 in \cite{chatterjee08}), we have $\pp(M-m \le -x) \le e^{-x^2/2\sigma^2}$. Combining all steps, we get 
\begin{align*}
\pp(\min_k X_{I^{(k)}} \le (1-\delta)M) &\le \pp(\min_k X_{I^{(k)}} \le M- \delta m/2) \\
&\quad + \pp(2M \le m)\\
&\le \pp\biggl(\sum_{k=1}^r (M-X_{I^{(k)}}) \ge \delta m/2\biggr) + e^{-m^2/8\sigma^2}\\
&\le \frac{8r\log |S|}{\delta m \beta} + \frac{2rt}{\delta} + e^{-m^2/8\sigma^2},
\end{align*}
and  
\begin{align*}
\pp(\max_{1\le k<l\le r} \rho_{I^{(k)}I^{(l)}} \ge \epsilon\sigma^2 ) &\le \pp\biggl(\sum_{1\le k<l\le r} \rho_{I^{(k)}I^{(l)}} \ge \epsilon \sigma^2\biggr) \\
&\le \frac{r^2}{2\epsilon \sigma^2} \ee\av{\rho}_{-t,t}. 
\end{align*}
Putting together the last two bounds, we see that the set $A:= \{I^{(1)}, \ldots, I^{(r)}\}$ satisfies the requirements of the theorem. 
\end{proof}
As a corollary of Theorem \ref{multigauss}, we now prove that multiple valleys exist at `all levels'. 
\begin{cor}\label{multigauss2}
Let all notation be the same as in Theorem \ref{multigauss}. Fix any $0 < \alpha \le 1$. Let 
\[
\gamma' := \gamma + 4e^{-\delta^2m^2/8\sigma^2} + 2e^{-m^2/8\sigma^2} + \frac{\sigma\sqrt{2\log r}}{\delta m}.
\]
Then with probability at least $1-\gamma'$, there exists a set $A\subseteq S$ of size $r$ such that for any distinct $i,j\in A$, we have $\rho_{ij}\le \epsilon \sigma^2$, and  for all $i\in A$, $|X_i - \alpha M|\le 5\delta |M|$. 
\end{cor}
\begin{proof}
Let $\bbx'$ be an independent copy of $\bbx$, and let
\[
\bby := \alpha \bbx + \sqrt{1-\alpha^2}\bbx'.
\]
Note that $\bby$ has the same distribution as $\bbx$. 
Let $A$ be a set as in Theorem~\ref{multigauss}. Let $M_Y := \max_i Y_i$. Then for any $i\in A$,
\begin{align*}
|Y_i - \alpha M_Y| &\le \alpha|X_i - M| + \alpha |M_Y-M| + \sqrt{1-\alpha^2} \max_{j\in A} |X_j'|\\
&\le \delta M + |M_Y-M| + \max_{j\in A}|X_j'|. 
\end{align*}
By Gaussian concentration (see e.g.\ Proposition 1.3 in \cite{chatterjee08}), we have
\[
\pp(|M_Y - M| > \delta m) \le 4e^{-\delta^2m^2/8\sigma^2}
\]
and 
\[
%\pp(|M| > 2m) \le 2e^{-m^2/2\sigma^2}, \ \ 
\pp(|M| < m/2) \le 2e^{-m^2/8\sigma^2}. 
\]
Moreover, by the independence of $\bbx'$ and $\bbx$ and a  standard result for Gaussian random variables (see e.g.\ Lemma 2.1 in \cite{chatterjee08}), we get 
\[
\ee(\max_{j\in A}|X_j'|) \le \ee\bigl(\sigma \sqrt{2\log |A|}\bigr)= \sigma \sqrt{2\log r}. 
\]
Therefore, 
\[
\pp(\max_{j\in A} |X_j'| > \delta m) \le \frac{\sigma\sqrt{2\log r}}{\delta m}. 
\]
From the above steps and Theorem \ref{multigauss}, we have that with probability at least $1-\gamma'$, there is a set $A\subseteq S$ of size $r$ such that for any distinct $i,j\in A$, we have $\rho_{ij}\le \epsilon \sigma^2$, and for each $i\in A$, $|Y_i - \alpha M_Y|\le 4\delta m$. 
Since $\bby$ and $\bbx$ have the same distribution, this completes the proof. 
\end{proof}

\subsection{Proofs of Theorem \ref{multisk} and Corollary \ref{multisk2}}
These are direct applications of Theorem \ref{multigauss} and Corollary \ref{multigauss2}. Consider the Gaussian field $(H_N(\bos))_{\bos\in \{-1,1\}^N}$ defined in \eqref{skhamil}, and choose
\begin{align*}
&\beta = e^{\sqrt{\log N}}, \ r = [(\log N)^{1/8}], \ \delta = (\log N)^{-1/8},  \\
&t = (\log N)^{-1/3}, \ \epsilon = e^{-(\log N)^{1/8}}.
\end{align*}
Note that $\sigma^2 = N/2$ and $|S|=2^N$. Note that the quantity $\rho_{\bos\bos'}$, according to the notation of Theorem \ref{chaosgauss}, is just $NR_{\bos, \bos'}^2$. Thus, with the above value of $t$, Theorem~\ref{chaoscont} says that for some absolute constants $C,c$, 
\[
\frac{1}{\sigma^2}\ee\av{\rho}_{-t,t} \le C N^{-c(\log N)^{-5/6}} = C e^{-c(\log N)^{1/6}}. 
\]
Again, by the Sudakov minoration technique (see e.g.\ Lemma 2.3 in \cite{chatterjee08}) it is not difficult to prove that  $m \ge c N$
for some positive absolute constant $c$. 
Invoking Theorem \ref{multigauss}, we now get
\begin{align*}
\gamma_N &\le C (\log N)^{1/4} e^{-\sqrt{\log N}} + C (\log N)^{-1/12} \\
&\qquad + Ce^{-cN} + C(\log N)^{1/4}e^{C(\log N)^{1/8}} e^{-c(\log N)^{1/6}}\\
&\le C(\log N)^{-1/12},
\end{align*}
where $C$ and $c$ denote arbitrary absolute constants. 
This completes the proof of Theorem \ref{multisk}. To prove Corollary \ref{multisk2}, note that the quantity $\gamma'$ in Corollary \ref{multigauss2} can be bounded by 
\[
\gamma_N + C e^{-cN (\log N)^{-1/4}} + CN^{-1/2} (\log N)^{1/8} \log \log N. 
\]
Since $m \ge cN$ as noted before, this completes the proof of Corollary \ref{multisk2}.

\subsection{Superconcentration in Gaussian fields}
Carrying on with the notation of Subsection \ref{chaosgauss}, we have the following formula for the variance of the free energy associated with a Gaussian field at inverse temperature $\beta$. This is a direct analog of Lemma 3.1 in \cite{chatterjee08}. We follow the notation of Subsection~\ref{chaosgauss}.
\begin{thm}\label{supergauss}
Take any $\beta \ge 0$. Let 
\[
F(\bbx):= \frac{1}{\beta}\log \sum_{i\in S} e^{\beta X_i}. 
\]
Then 
\[
\var F(\bbx) = \int_0^\infty e^{-t}\ee\av{\rho}_{0,t} dt.
\]
\end{thm}
\begin{proof}
Note that by Lemma \ref{basic}, for any smooth $f$ we have
\begin{align*}
\var(f(\bbx)) &= -\int_0^\infty \frac{d}{dt} \ee(f(\bbx^{-t}) f(\bbx^t)) dt\\
&= 2\int_0^\infty e^{-2t} \sum_{i,j}\rho_{ij}\ee(\partial_i f(\bbx^{-t}) \partial_j f(\bbx^t)) dt. 
\end{align*}
Taking $f = F$, we get
\[
\sum_{i,j}\rho_{ij}\partial_i F(\bbx^{-t}) \partial_j F(\bbx^t) = \av{\rho}_{-t,t}.  
\]
Again, since $(\bbx^{-t}, \bbx^t)$ has the same joint law as $(\bbx,\bbx^{2t})$, we see that $\ee\smallavg{\rho}_{-t,t} = \ee\smallavg{\rho}_{0,2t}$. Combining the steps, we get
\[
\var(F(\bbx)) = \int_0^\infty e^{-t}\ee\av{\rho}_{0,t} dt.
\]
This completes the proof.
\end{proof}

\subsection{Proof of Theorem \ref{superconc}}
This is just a combination of Theorem \ref{supergauss} above and Theorem \ref{chaoscont}.

\subsection{Chaos implies superconcentration}\label{nochaosproof} 
The goal of this subsection is to prove that in the absence of superconcentration, we do not have chaos either. This is an improved version of Theorem 3.2 in \cite{chatterjee08}, where the absence of chaos was proved only up to a finite time, but not for all $t$. 
\begin{lmm}\label{lowlmm}
Suppose $U$ is a nonnegative random variable and let $v := \ee((1+U)^{-1})$.  Then for any $t\ge 0$, 
$\ee(e^{-tU}) \ge \frac{1}{2}ve^{-t(2-v)/v}$. 
\end{lmm}
\begin{proof}
Note that
\begin{align*}
\ee(e^{-tU}) &= \int_0^1 \pp(e^{-tU} \ge y) dy\\
&= \int_0^1 \pp((1+U)^{-1} \ge (1-t^{-1}\log y)^{-1}) dy. 
\end{align*}
Now, for any $\epsilon > 0$, we have
\[
\ee((1+U)^{-1}) \le \epsilon + \pp((1+U)^{-1} \ge \epsilon). 
\]
Thus, if $\epsilon \le v/2$, then
\[
\pp((1+U)^{-1}\ge \epsilon) \ge \frac{v}{2}.
\]
Now $(1-t^{-1}\log y)^{-1} \le v/2$ if and only if $y \le e^{-t(2-v)/v}$. Combining the steps, we see that
\begin{align*}
\ee(e^{-tU}) \ge \int_0^{e^{-t(2-v)/v}} \frac{v}{2}dy = \frac{v}{2}e^{-t(2-v)/v}.
\end{align*}
This completes the proof of the lemma. 
\end{proof}
\begin{thm}\label{supchaos}
Let all notation be as in Subsection \ref{chaosgauss}. 
Let $F(\bbx)$ be as in Theorem \ref{supergauss}. 
Take any $\beta \in (0,\infty)$, and define 
\[
v := \frac{\var F(\bbx)}{\ee\av{\rho}_{0,0}}.
\]
Then for all $t \ge 0$,
\[
\ee\av{\rho}_{0,t}\ge \frac{1}{2}\var (F(\bbx))e^{-t(2-v)/v}. 
\]
\end{thm}
\begin{proof}
By Theorem \ref{supergauss},
\[
\var F(\bbx) = \int_0^\infty e^{-t}\ee\av{\rho}_{0,t} dt.
\]
By Lemma \ref{complete}, we see that there is a non-negative random variable $U$ such that for all $t$,
\[
\ee\av{\rho}_{0,t} = \ee\av{\rho}_{0,0}\ee(e^{-tU}). 
\]
Combined with the formula for the variance derived above, this gives
\[
\var F(\bbx) = \ee\av{\rho}_{0,0}\ee((1+U)^{-1}). 
\]
%If $\ee\smallavg{R}_{0,0} = 0$, the above identity implies that $\var(f(\bbx)) = 0$, and there is nothing to prove. 
The result now follows from Lemma \ref{lowlmm}.
\end{proof}
\subsection{A formula for the variance of Gaussian functionals}\label{plancherel}
In this subsection we present a general formula for the variance of a function of independent standard Gaussian random variables. After that, we derive a useful lower bound for the variance using this formula.

The variance formula looks similar to those in Hour\'e and Kagan~\cite{houdrekagan95} and Houdr\'e \cite{houdre95} but it is not the same. Various versions of the formula have appeared in Houdr\'e and P\'erez-Abreu (\cite{houdreperezabreu95}, Remark 2.3) and Houdr\'e, P\'erez-Abreu and Surgailis (\cite{houdreetal98}, Proposition 10). Essentially, this is the Parseval identity for the $L^2$ norm of a Gaussian functional expressed as a sum of squares of its Fourier coefficients in the orthogonal basis of multidimensional Hermite polynomials. We present a direct proof that does not involve the multivariate Hermite polynomial basis. Yet another  proof, based on heat kernel expansions, was suggested to the author in a private communication by Michel Ledoux. 
\begin{thm}\label{varformula}
Let $\bg = (g_1,\ldots, g_n)$ be a vector of i.i.d.\ standard Gaussian random variables, and let $f$ be a $C^\infty$ function of $\bg$ with bounded derivatives of all orders. Then 
\[
\var(f) = \sum_{k=1}^\infty \frac{1}{k!}\sum_{1\le i_1,\ldots,i_k\le n} \biggl(\ee\biggl(\frac{\partial^k f}{\partial g_{i_1}\cdots \partial g_{i_k}}\biggr)\biggr)^2. 
\]
The convergence of the infinite series is part of the conclusion. 
\end{thm}
\begin{proof}
Let $\bg'$  and $\bg''$ be i.i.d.\ copies of $\bg$, and for each $t\ge 0$, define
\[
\bg^t := \est \bg + \esst \bg', \ \  \bg^{-t} := \est \bg + \esst \bg''.
\]
Let
\[
\phi(t) := \ee(f(\bg^{-t}) f(\bg^t)). 
\]
Then by Lemma \ref{basic}, we have
\[
\phi'(t) = -2e^{-2t} \sum_i \ee(\partial_i f(\bg^{-t}) \partial_i f(\bg^t)). 
\]
For $0 < u \le 1$, define $\psi(u) = \phi(t(u))$, where $t(u) := -\frac{1}{2}\log u$. Then 
\begin{align*}
\psi'(u) &= \frac{d}{du} \ee(f(\bg^{-t(u)}) f(\bg^{t(u)}))\\
&=  -\frac{1}{2u}\phi'\biggl(-\frac{1}{2}\log u\biggr) \\
&= \sum_i \ee(\partial_i f(\bg^{-t(u)}) \partial_i f(\bg^{t(u)})).
\end{align*}
Repeating this step $k$ times shows that
\begin{align*}
\psi^{(k)}(u) = \sum_{1\le i_1,\ldots,i_k\le n} \ee(\partial_{i_1}\cdots \partial_{i_k} f (\bg^{-t(u)})\partial_{i_1}\cdots \partial_{i_k} f (\bg^{t(u)})). 
\end{align*}
As in the proof of Lemma \ref{complete}, we observe that the expectations on the right hand side are always nonnegative. We can continuously extend $\psi$ to the closed interval $[0,1]$ by defining $\psi(0) := \ee(f(\bg') f(\bg'')) = (\ee f(\bg))^2$. Then $\psi$ is a continuous function on $[0,1]$ that is $C^\infty$ in $(0,1)$ with all derivatives non-negative. Such functions are known as absolutely monotone (see Feller~\cite{feller71}, p.\ 223), and their most important property is that they can be represented as a power series $\psi(u) = \sum_{k=0}^\infty p_k u^k$, where the coefficients are non-negative and sum to $\psi(1)$. From this one can easily deduce that for any $k\ge 1$, 
\[
p_k = \lim_{u\ra 0} \frac{\psi^{(k)}(u)}{k!} = \frac{1}{k!}\sum_{1\le i_1,\ldots, i_k\le n} (\ee(\partial_{i_1}\cdots \partial_{i_k} f(\bg)))^2. 
\]
Since $p_0=\psi(0)=(\ee(f))^2$ and $\psi(1) = \ee(f^2)$, this completes the proof.
\end{proof}
A great advantage of Theorem \ref{varformula} is that we can extract lower bounds for the variance just by collecting a subset of the terms in the infinite sum. This is exactly what we do to get the following theorem. We do not actually need the theorem in its full generality (with respect to the smoothness conditions on $f$), but prove it in the general form nonetheless. 
\begin{thm}\label{varlowbd}
Suppose $f:\rr^n \ra \rr$ is an absolutely continuous function such that there is a version of its gradient $\nabla f$ that is bounded on bounded sets. Let $\bg$ be a standard Gaussian random vector in $\rr^n$, and suppose $\ee|f(\bg)|^2$ and $\ee|\nabla f(\bg)|^2$ are both finite. Then
\[
\var(f(\bg)) \ge \frac{1}{2}\sum_{i=1}^n \bigl(\ee(g_i \partial_i f(\bg))\bigr)^2 \ge \frac{1}{2n}\bigl(\ee(\bg \cdot \nabla f(\bg))\bigr)^2,
\]
where $\bx \cdot \by$ denotes the usual inner product on $\rr^n$.
\end{thm}
\begin{proof}
First assume that $f\in C^\infty_b$.
Theorem \ref{varformula} gives
\[
\var(f(\bg)) \ge \frac{1}{2}\sum_{i=1}^n (\ee(\partial_i^2 f(\bg)))^2. 
\]
Integration by parts gives 
\[
\ee(\partial_i^2 f(\bg)) = \ee((g_i^2 - 1) f(\bg)). 
\]
Thus, for any $C^\infty_b$ function $f$, 
\begin{equation}\label{varlow1}
\var(f(\bg)) \ge \frac{1}{2}\sum_{i=1}^n (\ee((g_i^2-1) f(\bg)))^2. 
\end{equation}
Let us now show that the above inequality holds for any bounded Lipschitz function $f$. For each $t>0$ and $\bx \in \rr^n$, define
\[
f_t(\bx) := \ee(f(\bx + t \bg)).
\]
Then we can write
\begin{align*}
f_t(\bx) &= \int_{\rr^n} f(\bx + t\by) \frac{e^{-\frac{1}{2}|\by|^2}}{(2\pi)^{n/2}} d\by\\
&= \int_{\rr^n} t^{-n} f(\bz) \frac{e^{-\frac{1}{2t^2}|\bz-\bx|^2}}{(2\pi)^{n/2}} d\bz. 
\end{align*}
Since $f$ is a bounded function, it is clear from the above representation that $f_t\in C^\infty_b$ for any $t>0$, and hence \eqref{varlow1} holds for $f_t$. Again, since $f$ is Lipschitz, 
\[
|f_t(\bx) - f(\bx) |\le Lt \ee|\bg|,
\]
where $L$ is the Lipschitz constant of $f$. This shows that we can take $t\ra 0$ and obtain \eqref{varlow1} for $f$. 

Next, we want to show \eqref{varlow1} whenever $f$ is absolutely continuous and square-integrable under the Gaussian measure, and the gradient of  $f$ is bounded on bounded sets. Take any such $f$. Let $h:\rr^n \ra [0,1]$ be a Lipschitz function that equals $1$ in the ball of radius $1$ centered at the origin, and vanishes outside the ball of radius $2$. For each $n\ge 1$, define 
\[
f_n(\bx) := f(\bx) h(n^{-1}\bx). 
\]
Then note that each $f_n$ is bounded and Lipschitz (with possibly increasing Lipschitz constants). Thus, \eqref{varlow1} holds for each $f_n$. Since $|f_n|\le |f|$ everywhere, and $f_n \ra f$ pointwise, and $f$ is square-integrable under the Gaussian measure, it follows that we can take $n\ra \infty$ and get \eqref{varlow1} for $f$. 

Finally, we wish to show that if $\nabla f$ is square-integrable under the Gaussian measure, we have
\[
\ee ((g_i^2-1)f(\bg)) = \ee(g_i\partial_i f(\bg)).
\]
(Note that $f$ is almost surely an absolutely continuous function of $g_i$ if we fix $(g_j)_{j\ne i}$. This follows from Fubini's theorem.) The above identity follows from the univariate identity 
\[
\ee(Z\phi(Z)) = \ee(\phi'(Z))
\]
that holds when $Z$ is a standard Gaussian random variable and $\phi$ is any absolutely continuous function such that $\ee|\phi(Z)|$, $\ee|Z\phi(Z)|$ and $\ee|\phi'(Z)|$ are all finite. The identity is just integration-by-parts when $\phi$ is absolutely continuous and vanishes outside a bounded set. In the general case, let $\phi_n(x)= \phi(x)h(x/n)$, where $h:\rr\ra [0,1]$ is a Lipschitz function that is $1$ on $[-1,1]$ and vanishes outside $[-2,2]$. Then the above identity holds for each $\phi_n$, and we can pass to the limit using the dominated convergence theorem. (Actually, it can be shown that the finiteness of $\ee|\phi'(Z)|$ suffices.) As a last step, we observe that by the Cauchy-Schwarz inequality,
\[
\sum_{i=1}^n (\ee(g_i \partial_i f(\bg)))^2 \ge \frac{1}{n}\biggl(\sum_{i=1}^n \ee(g_i\partial_i f(\bg))\biggr)^2. 
\]
This completes the proof.
\end{proof}

\subsection{Proof of Theorem \ref{nosuper}} Let $\bg$ be as in the previous subsection.  
Let $A$ be a finite subset of $\rr^n$. Consider the function
\[
f_\beta(\bx) := \frac{1}{\beta}\log \sum_{\by\in A} e^{\beta \by\cdot \bx}. 
\]
\begin{lmm}\label{varlmm}
For any $\beta > 0$ we have
\[
\var(f_\beta(\bg)) \ge \sup_{0\le \beta'\le \beta} \frac{{\beta'}^2}{2n}\biggl(\sum_{i=1}^n\ee \biggl(\frac{\sum_{\by\in A} y_i^2 e^{\beta' \by\cdot\bx}}{\sum_{\by\in A}e^{\beta'\by \cdot \bx}} - \biggl(\frac{\sum_{\by\in A} y_i e^{\beta' \by\cdot\bx}}{\sum_{\by\in A}e^{\beta'\by \cdot \bx}}\biggr)^2\biggr)\biggr)^2.
\]
\end{lmm}
\begin{proof}
Note that
\[
\partial_i f_\beta(\bx)  = \frac{\sum_{\by\in A} y_i e^{\beta \by\cdot \bx}}{\sum_{\by\in A} e^{\beta \by\cdot \bx}},
\]
and therefore 
\[
\bx \cdot \nabla f_\beta(\bx) = \frac{\sum_{\by\in A} (\by\cdot \bx) e^{\beta \by\cdot \bx}}{\sum_{\by\in A} e^{\beta \by\cdot \bx}} = \fpar{}{\beta} \log \sum_{\by\in A} e^{\beta \by \cdot \bx}. 
\]
Now, it is easy to verify that $\log \sum e^{\beta \by\cdot \bx}$ is a convex function of $\beta$, and hence for each $\bx$, $\bx \cdot \nabla f_\beta(\bx)$ is an increasing function of $\beta$. %When $\beta=\infty$, almost every $\bx \in \rr^n$ has the property that a unique $\by\in A$ maximizes $\bx \cdot \by$. For any such $\bx$, we can set $\nabla f(\bx)=\by$, where $\by$ is the unique maximizer of~$\by\cdot \bx$. This gives a version of $\nabla f$ that is bounded on bounded sets. 
%Moreover, it easy to see that almost everywhere, 
%\[
%\lim_{\beta \ra \infty } \bx \cdot \nabla f_\beta (\bx) = \bx \cdot \nabla f_\infty (\bx). 
%\]
Thus, $\ee(\bg \cdot \nabla f_\beta (\bg))$ is also an increasing function of $\beta$. Moreover,
\[
\ee(\bg \cdot \nabla f_0(\bg)) = \frac{1}{|A|}\sum_{\by\in A} \ee(\by \cdot \bg) = 0, 
\]
and therefore $\ee(\bg \cdot \nabla f_\beta (\bg))\ge 0$ for all $\beta>0$. Finally note that by integration by parts,
\begin{align*}
\ee(\bg \cdot\nabla f_\beta (\bg)) &= \sum_{i=1}^n \ee(\partial_i^2 f_\beta(\bg)) \\
&= \beta \biggl(\frac{\sum_{\by\in A} y_i^2 e^{\beta \by\cdot\bx}}{\sum_{\by\in A}e^{\beta\by \cdot \bx}} - \biggl(\frac{\sum_{\by\in A} y_i e^{\beta \by\cdot\bx}}{\sum_{\by\in A}e^{\beta\by \cdot \bx}}\biggr)^2\biggr).
\end{align*}
Combined with Theorem \ref{varlowbd}, this completes the proof.
\end{proof}

We are now ready to complete the proof of Theorem \ref{nosuper}. 
Consider an undirected graph $G = (V,E)$, and the Edwards-Anderson spin glass model on $G$ as defined in Subsection \ref{ea}. Let $\smallavg{\cdot}_\beta$ denote the average with respect to the Gibbs measure at inverse temperature $\beta$. First, we will work with $\beta < \infty$.  Let $F$ be as in Theorem \ref{nosuper}. By Lemma \ref{varlmm}, with  $n = |E|$, $\bg = (g_{ij})_{(i,j)\in E}$ and $A = \{(\sigma_i\sigma_j)_{(i,j)\in E}: \bos\in \{-1,1\}^V\}$,  we get 
\[
\var F(\beta) \ge \sup_{0\le \beta'\le \beta}\frac{\beta'^2}{2|E|} \biggl(|E| - \sum_{(i,j)\in E} \ee\av{\sigma_i\sigma_j}^2_{\beta'}\biggr)^2. 
\]
Now, under the Gibbs measure at inverse temperature $\beta'$, the conditional expectation of $\sigma_i$ given the rest of the spins is $\tanh(\beta'\sum_{j\in N(i)} g_{ij}\sigma_j)$, where $N(i)$ is the neighborhood of $i$ in the graph $G$. Using this fact and the inequality $|\tanh x|\le |x|$, we get
\begin{align*}
\ee\av{\sigma_i\sigma_j}^2_{\beta'} &= \ee\bigavg{\tanh\biggl(\beta'\sum_{k\in N(i)} g_{ik}\sigma_k\biggr)\sigma_j}^2\\
&\le {\beta'}^2\ee\biggl(\sum_{k\in N(i)}|g_{ik}|\biggr)^2\\
&\le {\beta'}^2 d\sum_{k\in N(i)} \ee|g_{ik} |^2\le {\beta'}^2 d^2. 
\end{align*}
Thus, 
\[
\var F(\beta) \ge \frac{|E|}{2}\sup_{0\le \beta' \le \min\{\beta, 1/d\}} \beta'^2(1-{\beta'}^2 d^2)^2. 
\]
Taking $\beta' = \min\{\beta,1/2d\}$, we get 
\[
\var F(\beta) \ge \frac{9|E|}{32}\min\biggl\{\beta^2, \frac{1}{4d^2}\biggr\}. 
\]
Finally, to prove the lower bound for $\beta = \infty$, just note that $F(\beta) \ra F(\infty)$ almost surely, and the quantities are all bounded, so we can apply the dominated convergence theorem to get convergence of the variance. This completes the proof of Theorem \ref{nosuper}.

\subsection{Proof of Theorem \ref{nochaos}}
For $\beta < \infty$, this is just a combination of   Theorem \ref{nosuper} and Theorem \ref{supchaos}. (Note that the notations of the two theorems are related as $\ee\smallavg{\rho}_{0,t} = |E|\ee(Q_{1,2})$; also note that $\ee\smallavg{\rho}_{0,0}\le |E|$ in this case, and therefore $v \ge Cq$.) 

Next, note that as $\beta \ra \infty$, the Gibbs measure at inverse temperature $\beta$ converges weakly to the uniform distribution on the set of ground states. The same holds for the perturbed Gibbs measure.  Thus, 
\[
\lim_{\beta \ra \infty} \av{Q_{1,2}}_\beta = \av{Q_{1,2}}_\infty \ \ \text{a.s.},
\]
where $\smallavg{Q_{1,2}}_\beta$ denotes the Gibbs average at inverse temperature $\beta$. Since all quantities are bounded by $1$, we can take expectations on both sides and apply dominated convergence.

\subsection{Proof of Theorem \ref{quenchedchaos}}\label{quenchedproof}
Let $\bg = (g_{ij})_{(i,j)\in E}$, and let $\bg', \bg''$ be independent copies of $\bg$. For each $t$, let 
\[
\bg^t := \est \bg + \esst \bg', \ \ \ \bg^{-t} := \est \bg + \esst \bg''.
\]
For each $t\in \rr$, let $\bos^t$ denote a configuration drawn from the Gibbs measure defined by the disorder $\bg^t$. For $t\ne s$, we assume that $\bos^t$ and $\bos^s$ are independent given $\bg, \bg', \bg''$. Define
\[
\phi(t) := \frac{1}{|E|} \sum_{(i,j)\in E}\ee\bigl(\av{\sigma_i^t\sigma_j^t}\av{\sigma_i^{-t}\sigma_j^{-t}}). 
\]
By Lemma \ref{complete}, it follows that $\phi$ is a completely monotone function on $[0,\infty)$. Also, $\phi$ is bounded by $1$. Thus, for any $t > 0$,
\begin{equation}\label{pbd}
|\phi'(t)| \le \frac{\phi(0)-\phi(t)}{t} \le \frac{1}{t}. 
\end{equation}
Again, if we let 
\[
e_{ijkl}^t := \av{\sigma_i^t \sigma_j^t \sigma_k^t \sigma_l^t} - \av{\sigma_i^t \sigma_j^t}\av{\sigma_k^t \sigma_l^t}, 
\]
then by Lemma \ref{basic} we know that
\begin{align*}
\phi'(t) = -\frac{2e^{-2t}\beta^2}{|E|} \sum_{(i,j)\in E, \, (k,l)\in E} \ee(e_{ijkl}^t e_{ijkl}^{-t}). 
\end{align*}
Now fix $t$, and let $\bar{e}_{ijkl} := \ee(e_{ijkl}^t\mid \bg)$. Then $\ee(e_{ijkl}^t e_{ijkl}^{-t}) = \ee(\bar{e}_{ijkl}^2)$ and so by~\eqref{pbd}, we have
\begin{equation}\label{ffbd}
\sum_{(i,j)\in E, \, (k,l)\in E} \ee(\bar{e}_{ijkl}^2)\le \frac{|E|}{2te^{-2t}\beta^2}. 
\end{equation}
Now let
\[
u_{ijkl}^t := \av{\sigma_i^t \sigma_j^t \sigma_k^t \sigma_l^t}, \ \ \ v_{ijkl}^t := \av{\sigma_i^t \sigma_j^t}\av{\sigma_k^t \sigma_l^t},
\]
and define $\bar{u}_{ijkl} := \ee(u_{ijkl}^t\mid \bg)$ and $\bar{v}_{ijkl} := \ee(v_{ijkl}\mid \bg)$. Then $|\bar{u}_{ijkl}|$ and $|\bar{v}_{ijkl}|$ are both uniformly bounded by $1$, and so
\begin{align*}
\sum_{(i,j)\in E, \, (k,l)\in E} \ee(u_{ijkl}^t u_{ijkl}^{-t} - v_{ijkl}^t v_{ijkl}^{-t}) &= \sum_{(i,j)\in E, \, (k,l)\in E} \ee(\bar{u}_{ijkl}^2 - \bar{v}_{ijkl}^2) \\
&\le 2 \sum_{(i,j)\in E, \, (k,l)\in E} \ee|\bar{u}_{ijkl} - \bar{v}_{ijkl}|.  
\end{align*}
Since $\bar{u}_{ijkl} - \bar{v}_{ijkl} = \bar{e}_{ijkl}$, an application of the Cauchy-Schwarz inequality and \eqref{ffbd} to the above bound gives 
\begin{align*}
\sum_{(i,j)\in E, \, (k,l)\in E} \ee(u_{ijkl}^t u_{ijkl}^{-t} - v_{ijkl}^t v_{ijkl}^{-t}) &\le \frac{2|E|^{3/2}}{e^{-t}\beta \sqrt{2t}}. 
\end{align*}
To complete the proof, note that
\begin{align*}
\frac{1}{|E|^2}\sum_{(i,j)\in E, \, (k,l)\in E} (u_{ijkl}^t u_{ijkl}^{-t} - v_{ijkl}^t v_{ijkl}^{-t}) &=  \av{(Q_{\bos^t,\bos^{-t}} - \av{Q_{\bos^t, \bos^{-t}}})^2},
\end{align*}
where $Q_{\bos^t, \bos^{-t}}$ is the bond overlap between $\bos^t$ and $\bos^{-t}$. 

\subsection{Chaos under discrete perturbation}\label{chaosdiscproof}
Our goal in this subsection is to prove that superconcentration implies chaos under discrete perturbations. Accordingly, let us first set the stage for discrete perturbation. Henceforth, we deviate from the notation of Subsection \ref{chaosgauss}.

The result of this subsection and its proof are inspired by Lemma 2.3 in~\cite{chatterjee08a}; we follow the same notation as in \cite{chatterjee08a}. Let $\bbx = (X_1,\ldots,X_n)$ be a vector of independent  random variables with $\var(X_i) = 1$ for each $i$. Let $\bbx'$ be an independent copy of $\bbx$. For any $A\subseteq [n] := \{1,\ldots,n\}$, let $\bbx^A$ be the vector whose $i$th component is
\[
X_i^A := 
\begin{cases}
X_i' &\text{ if } i\in A,\\
X_i &\text{ if } i\not\in A.
\end{cases}
\]
Let $f:\rr^n \ra \rr$ be a twice differentiable function. Let $\partial_i f$ and $\partial_i^2 f$ be the first and second partial derivatives of $f$ in the direction of the $i$th coordinate. 
\begin{thm}\label{discgauss}
Suppose $\epsilon$ and $\delta$ are constants such that for all $i$, $|\partial_i f|\le \delta$ and $|\partial_i^2 f|\le \epsilon$ everywhere in the closed convex hull of the support of $\bbx$. Fix $0\le k\le n$, and let $A$ be a subset of~$[n]$, chosen uniformly at random from the collection of all subsets of size~$k$. Define $\bbx^A$ as above. Let $\gamma := \max_i \ee|X_i - X_i'|^3$. Then
\begin{equation}\label{discgaussineq}
\ee\biggl(\sum_{i=1}^n \partial_i f(\bbx)\partial_i f(\bbx^A)\biggr) \le \frac{n+1}{k+1}\var(f(\bbx)) + \frac{3n\delta\epsilon\gamma}{2}. 
\end{equation}
\end{thm}
The proof of Theorem \ref{discgauss} is divided into a series of lemmas. First, let us introduce some further conventions. To simplify notation, we will write $f^A$ for $f(\bbx^A)$. When $A=\emptyset$, we will simply write $f$. For any $i$ and $A$ such that $i\not \in A$, let 
\[
\Delta_i f^A :=  f^A - f^{A\cup\{i\}}. 
\]
As usual, when $A=\emptyset$, we will simply write $\Delta_i f$. Let $\ma_{k,i}$ denote the collection of all subsets of $[n]\backslash \{i\}$ of size $k$. For $0\le k\le n-1$, efine
\[
T_k := \sum_{i=1}^n \frac{1}{{n-1\choose k}}\sum_{A\in \ma_{k,i}} \ee(\Delta_i f\Delta_i f^A), 
\]
The above quantity is a discrete proxy for the left hand side in \eqref{discgaussineq}. Our first result is an exact formula for the variance in terms of $T_0,\ldots,T_{n-1}$. This is actually a restatement of Lemma 2.3 from \cite{chatterjee08a}. 
\begin{lmm}\label{disc1}
We have
\[
\var(f) = \frac{1}{2n}\sum_{k=0}^{n-1} T_k.
\]
\end{lmm}

\begin{proof}
By exchangeability of $X_i$ and $X_i'$, it is easy to see that the pair $(f, \Delta_i f^A)$ has the same distribution as the pair $(f^{\{i\}}, -\Delta_i f^A)$, and therefore
\begin{equation}\label{dirichlet}
\ee(\Delta_i f \Delta_i f^A) = \ee(f\Delta_i f^A) - \ee(f^{\{i\}} \Delta_i f^A) = 2\ee(f\Delta_i f^A). 
\end{equation}
We claim that 
\begin{equation}\label{tree}
\frac{1}{n}\sum_{i=1}^n \sum_{k=0}^{n-1}\frac{1}{{n-1 \choose k}} \sum_{A\in \ma_{k,i}} \Delta_i f^A = f - f^{[n]}. 
\end{equation}
To see this, consider any $B\subseteq [n]$ such that $B\neq \emptyset$ and $B\neq [n]$. Let $k= |B|$. On the left hand side in the above display, if we write out the definition of $\Delta_i f^A$ as $f^A - f^{A\cup\{i\}}$ and regroup terms, then the coefficient of $f^B$ in the expansion is
\[
\frac{1}{n{n-1\choose k}}(n-k) - \frac{1}{n{n-1 \choose k-1}}k = 0. 
\]
Similarly, the coefficient of $f$ is $1$ and the coefficient of $f^{[n]}$ is $-1$. This proves \eqref{tree}. Combining \eqref{tree} with \eqref{dirichlet}, we see that
\[
\var(f) = \ee(f(f-f^{[n]})) = \frac{1}{2n}\sum_{i=1}^n \sum_{k=0}^{n-1}\frac{1}{{n-1 \choose k}} \sum_{A\in \ma_{k,i}} \ee(\Delta_i f\Delta_i f^A). 
\]
This completes the proof of the lemma.
\end{proof}
Our next lemma is a monotonicity property of the $T_k$'s.
\begin{lmm}\label{disc2}
We have $T_0\ge T_1\ge \cdots \ge T_{n-1}\ge 0$. 
\end{lmm}
\begin{proof}
Take any $A$ and $i\not \in A$. It is easy to see that given $(X_j)_{j\not \in A}$ and $X_i'$, the random variables $\Delta_i f$ and $\Delta_i f^A$ are i.i.d. Therefore, 
\begin{align*}
\ee(\Delta_i f \Delta_i f^A) &= \ee((\ee(\Delta_i f\mid (X_j)_{j\not \in A}, X_i'))^2).
%\\
%&= \ee((g_A - g_A^i)^2) \\
%&= 2\ee(\var(g_A \mid (X_j)_{j\not \in A\cup\{i\}})). 
\end{align*}
From this and Jensen's inequality, it is clear that $\ee(\Delta_i f \Delta_i f^A) \ge 0$, and for any $A\subseteq B\subseteq [n]\backslash \{i\}$, 
\[
\ee(\Delta_i f \Delta_i f^A) \ge \ee(\Delta_i f \Delta_i f^B).
\]
Thus, if $k := |A| \le n-2$, we have
\[
\ee(\Delta_i f \Delta_i f^A) \ge \frac{1}{n-k-1}\sum\ee(\Delta_i f \Delta_i f^B),
\]
where the sum is taken over all $B$ such that $B = A\cup \{j\}$ for some $j \not \in A \cup \{i\}$. Since any $B\in \ma_{k+1,i}$ can be obtained by adding one element to $A$  for exactly $k+1$ many $A\in \ma_{k,i}$, we have
\[
\sum_{A\in \ma_{k,i}}\ee(\Delta_i f \Delta_i f^A) \ge \frac{k+1}{n-k-1}\sum_{B\in \ma_{k+1,i}}\ee(\Delta_i f \Delta_i f^B).
\]
This can be rewritten as
\[
\frac{1}{{n-1 \choose k}}\sum_{A\in \ma_{k,i}}\ee(\Delta_i f \Delta_i f^A) \ge \frac{1}{{n-1\choose k+1}}\sum_{B\in \ma_{k+1,i}}\ee(\Delta_i f \Delta_i f^B).
\]
This completes the proof of the lemma.
\end{proof}
Combining Lemma \ref{disc1} and Lemma \ref{disc2}, we easily get the following discrete version of Theorem \ref{discgauss}.
\begin{lmm}\label{disc3}
For each $0\le k\le n-1$,
\begin{equation*}
T_k \le  \frac{2n\var(f)}{k+1}.
\end{equation*}
\end{lmm}
\begin{proof}
Since $T_0\ge T_1\ge \cdots T_{n-1}\ge 0$, and 
\[
\var(f) = \frac{1}{2n}\sum_{k=0}^{n-1} T_k,
\]
it follows that for each $0\le k\le n-1$,
\begin{equation*}\label{tkbd}
T_k \le \frac{1}{k+1}\sum_{r=0}^k T_r \le \frac{2n\var(f)}{k+1}.
\end{equation*}
This completes the proof of the lemma.
\end{proof}
Finally, we are ready to prove Theorem \ref{discgauss}. This involves replacing the discrete derivatives in Lemma \ref{disc3} with continuous derivatives, and incurring a small error along the way.
\begin{proof}[Proof of Theorem \ref{discgauss}]
Since $|\partial_i f|\le \delta$ and $|\partial_i^2 f|\le \epsilon$ everywhere on the closed convex hull of the  support of $\bbx$, by Taylor expansion we have
\[
|\Delta_i f^A|\le |X_i - X_i'|\delta, \ \ |\Delta_i f^A - (X_i - X_i')\partial_i f^A| \le \frac{\epsilon}{2}(X_i - X_i')^2. 
\]
Thus,
\begin{equation}\label{discrete}
\begin{split}
&|\ee(\Delta_i f\Delta_i f^A) - \ee((X_i - X_i')^2\partial_i f\partial_i f^A)| \\
&\le |\ee((\Delta_i f - (X_i - X_i') \partial_i f) \Delta_i f^A)| \\
&\qquad + |\ee((X_i - X_i')\partial_i f (\Delta_i f^A - (X_i - X_i')\partial_i f^A))|\\
&\le \delta\epsilon\ee|X_i - X_i'|^3.
\end{split}
\end{equation}
Now let $X_i''$ be another independent copy of $X_i$, that is also independent of~$X_i'$. Let $\widetilde{\partial_i f}$ denote $\partial_i f$ with $X_i$ replaced by $X_i''$ and define $\widetilde{\partial_i f}^A$ similarly. Since $\var(X_i) = 1$ and $(X_i-X_i')^2$ is independent of $\widetilde{\partial_i f}\widetilde{\partial_i f}^A$, we have 
\[
\ee((X_i - X_i')^2 \widetilde{\partial_i f}\widetilde{\partial_i f}^A) = 2\;\ee(\widetilde{\partial_i f}\widetilde{\partial_i f}^A) = 2\; \ee(\partial_i f\partial_i f^A).
\]
Again,
\begin{align*}
|\partial_i f \partial_i f^A - \widetilde{\partial_i f}\widetilde{\partial_i f}^A| &\le 2\delta\epsilon |X_i - X_i''|. 
\end{align*}
Combining the last two observations, we get
\begin{align*}
|\ee((X_i - X_i')^2 \partial_i f \partial_i f^A) - 2\;\ee(\partial_i f\partial_i f^A)| &\le 2\delta \epsilon \ee((X_i- X_i')^2 |X_i - X_i''|) \\
&\le 2\delta \epsilon \ee|X_i - X_i'|^3. 
\end{align*}
And now, combining the above bound with \eqref{discrete}, we have
\begin{equation}\label{mainbd}
2\; \ee(\partial_i f \partial_i f^A) \le \ee(\Delta_i f \Delta_i f^A) + 3\delta \epsilon\ee|X_i - X_i'|^3. 
\end{equation}
We also have to consider the case when $i\in A$. Let $B = A\backslash \{i\}$. Then by Jensen's inequality we have
\begin{equation}\label{include}
\begin{split}
\ee(\partial_i f\partial_i f^A) &= \ee((\ee(\partial_i f\mid (X_j)_{j\not \in A}))^2)\\
&\le \ee((\ee(\partial_i f\mid (X_j)_{j\not \in B}))^2) = \ee(\partial_i f \partial_i f^B). 
\end{split}
\end{equation}
Now take $1\le k\le n-1$ and let $\ma_k$ denote the set of all subsets of $[n]$ of size $k$. Using \eqref{mainbd} and \eqref{include}, we get  
\begin{align*}
\sum_{i=1}^n \sum_{A\in \ma_k} \ee(\partial_i f \partial_i f^A) &= \sum_{i=1}^n \biggl(\sum_{A\in \ma_{k,i}} \ee(\partial_i f \partial_i f^A) +  \sum_{A\in \ma_{k-1,i}} \ee(\partial_i f \partial_i f^{A\cup \{i\}})\biggr)\\
&\le \sum_{i=1}^n \biggl(\sum_{A\in \ma_{k,i}} \ee(\partial_i f \partial_i f^A) +  \sum_{A\in \ma_{k-1,i}} \ee(\partial_i f \partial_i f^{A})\biggr)\\
&\le \frac{1}{2}{n-1 \choose k}T_k + \frac{1}{2}{n-1 \choose k-1} T_{k-1} + \frac{n}{2}{n\choose k} 3\delta\epsilon \gamma. 
\end{align*}
From this and Lemma \ref{disc3}, we conclude that for $1\le k\le n-1$, 
\begin{align*}
\frac{1}{{n\choose k}}\sum_{i=1}^n \sum_{A\in \ma_k} \ee(\partial_i f \partial_i f^A) &\le \frac{n-k}{2n}T_k + \frac{k}{2n} T_{k-1} + \frac{3n\delta \epsilon\gamma}{2}  \\
&\le \frac{n+1}{k+1} \var(f) + \frac{3n\delta \epsilon\gamma}{2}.
\end{align*}
The same conclusion can be drawn for $k=0$ and $k=n$ by defining $T_{-1}=T_n = 0$ and verifying that all steps hold. This completes the proof. 
\end{proof}
\subsection{Proof of Theorem \ref{chaosdisc}} 
Consider the S-K Hamiltonian $H_N$ defined in \eqref{skhamil} as a function of the disorder $\bg = (g_{ij})_{1\le i,j\le N}$. Fix $\beta$, and define $f = N^{-1/2}F_N(\beta)$, where $F_N(\beta)$ is the free energy defined in \eqref{free}. Let $\bg'$ be an independent copy of $\bg$, and define $\bg^A$ as we defined $\bbx^A$ in Theorem~\ref{discgauss}. Let $k = pN$ (and assume that $k$ is an integer), and define a perturbed Hamiltonian using the disorder $\bg^A$, where $A$ is chosen uniformly at random from the set of all subsets of $\{(i,j)\}_{1\le i,j\le N}$ of size $k$. 

Let $\bos^1$ be sampled from the original Gibbs measure, and $\bos^2$ from the perturbed Gibbs measure. An easy verification shows that
\[
\sum_{i,j} \partial_{ij} f(\bg) \partial_{ij}f(\bg^A) = \av{R_{1,2}^2},
\]
where $\partial_{ij}f$ is the derivative of $f$ with respect to the $(i,j)$th coordinate. On the other hand, by Theorem \ref{superconc} we know that
\[
\var f(\bg) \le \frac{C\log(2+C \beta)}{\log N}. 
\]
Finally, note that for any $(i,j)$, 
\[
\partial_{ij} f = \frac{\av{\sigma_i\sigma_j}}{N}, \ \ \partial_{ij}^2 f = \frac{\beta(1- \av{\sigma_i \sigma_j}^2)}{N^{3/2}}. 
\]
Therefore, we can take $\delta = N^{-1}$ and $\epsilon = \beta N^{-3/2}$ while applying Theorem~\ref{discgauss}. Using all the above information, we can now apply Theorem \ref{discgauss} to conclude that
\[
\ee\av{R_{1,2}^2} \le \frac{C\log(2+C \beta)}{p\log N} + C\beta N^{-1/2},
\]
where $C$ is an absolute constant. Since $p\in (0,1)$, we can ignore the second term on the right after replacing $\log(2+C \beta)$ by $C(1+\beta)$ in the first term. This completes the proof.

\subsection{Sharpness of Theorem \ref{main} for the REM}\label{rem}

The Random Energy Model (REM), introduced by Derrida \cite{derrida80, derrida81}, is possibly the simplest model of a spin glass. The state space is $\{-1,1\}^N$ as usual, but here the energies of states $\{-H_N(\bos)\}_{\bos\in \{-1,1\}^N}$ are chosen to be i.i.d.\ Gaussian random variables with mean zero and variance $N$. We show that Theorem~\ref{main} gives a  sharp result in the low temperature regime ($\beta > 2\sqrt{\log 2}$) of this model. We follow the notation of Theorem \ref{main}. 
\begin{prop}
Suppose $\bos^1$ is drawn from the original Gibbs measure of the REM and $\bos^2$ from the Gibbs measure perturbed continuously up to time~$t$, in the sense of Subsection \ref{chaossec}. If $\beta > 2\sqrt{\log 2}$, there are positive constants $C(\beta)$ and $c(\beta)$ depending only on $\beta$ such that for all $N$ and $t$,
\[
c(\beta) e^{-C(\beta)N\min\{1, t\}}\le \ee\av{1_{\{\bos^1 = \bos^2\}}}_{0,t}\le C(\beta) e^{-c(\beta)N\min\{1, t\}}.
\] 
\end{prop}
\begin{proof}
In the  notation of Theorem \ref{main}, we have $\rho_{\bos\bos'} = 0$ if $\bos \ne \bos'$, and $\rho_{\bos\bos'} = N$ if $\bos = \bos'$. Also, clearly, $\nu_\bos = 2^{-N}$ for each $\bos$. Suppose $\bos^1$ is drawn from the original Gibbs measure and $\bos^2$ from the Gibbs measure perturbed continuously up to time~$t$. Taking $\phi(x) = x/N$ in Theorem \ref{main}, we get
\begin{align*}
\ee\av{1_{\{\bos^1 = \bos^2\}}}_{0,t} &\le \inf_{s\ge t} \bigl(2^{-N} e^{2\beta^2 e^{-s} N}\bigr)^{t/s}.
\end{align*}
Now choose $s$ so large that $2\beta^2 e^{-s} \le \frac{1}{2}\log 2$. The above inequality shows that for $t\le s$,
\begin{align}\label{upbd1}
\ee\av{1_{\{\bos^1 = \bos^2\}}}_{0,t} &\le 2^{-Nt/2s},
\end{align}
and for $t> s$,
\begin{align}\label{upbd2}
\ee\av{1_{\{\bos^1 = \bos^2\}}}_{0,t} &\le 2^{-N} e^{2\beta^2 e^{-t} N}
\end{align}
A simple computation via Theorem \ref{supergauss} now gives
\[
\var( F_N(\beta) )\le C(\beta),
\]
where $C(\beta)$ is a constant depending only on $\beta$. 

Now suppose $\beta > 2\sqrt{\log 2}$. Let $H_N'(\bos) = H_N(\bos) + Na_N$, where $a_N$ solves
\[
Na_N^2 = \log \biggl(\frac{2^N}{\sqrt{N}}\biggr). 
\]
Let $(w^N_\alpha)_{1\le \alpha \le 2^N}$ denote the numbers $\exp(-\beta H_N'(\bos))$ when enumerated in non-increasing order. It follows from arguments in Section 1.2 of Talagrand~\cite{talagrand03} that this point process converges in distribution, as $N\ra \infty$, to a Poisson point process $(w_\alpha)_{\alpha \ge 1}$ with intensity $x^{-m-1}$ on $[0,\infty)$, where $m = 2\sqrt{\log 2}/\beta$. It is not difficult to extend this argument to show that 
\[
\lim_{N\ra \infty}\var\biggl(\log \sum_{\alpha =1}^{2^N} w_\alpha^N\biggr) = \var\biggl(\log \sum_{\alpha=1}^\infty w_\alpha\biggr)  > 0.
\]
We skip the details, which are somewhat tedious. (Here $\beta > 2\sqrt{\log 2}$ is required to ensure that the infinite sum $\sum_1^\infty w_\alpha$ converges almost surely.) 

However, $\var(\log \sum w_\alpha^N) = \var (\beta F_N(\beta))$. Thus, there is a positive constant $c(\beta)$ depending only on $\beta$ such that for any $N$, 
\[
\var(F_N(\beta)) \ge c(\beta).
\]
We can now use Theorem \ref{supchaos} to prove that for some positive constant $c(\beta)$ depending only on $\beta$, we have that for any $N$ and $t$,
\begin{align}\label{upbd3}
\ee\av{1_{\{\bos^1 = \bos^2\}}}_{0,t} &\ge c(\beta)e^{-Nt/c(\beta)}. 
\end{align}
However, we also have by Theorem \ref{main} that $\ee\av{1_{\{\bos^1 = \bos^2\}}}_{0,t}$ is a decreasing function of $t$, and hence
\[
\ee\av{1_{\{\bos^1 = \bos^2\}}}_{0,t} \ge \ee\av{1_{\{\bos^1 = \bos^2\}}}_{0,\infty} = 2^{-N}. 
\]
Combined with \eqref{upbd1}, \eqref{upbd2} and \eqref{upbd3}, this completes the proof.
\end{proof}

\vskip.2in
\noindent{\bf Acknowledgments.} The author thanks Michel Talagrand, Persi Diaconis, Daniel Fisher, Victor P\'erez-Abreu, Christian Houdr\'e, Michel Ledoux, Rongfeng Sun,  Tonci Antunovic  and Partha Dey for helpful discussions and comments, and Itai Benjamini for asking the question that led to Theorem~\ref{chaosdisc}.

\end{document}